\documentclass{article}

\usepackage{booktabs, multirow, bm}
\usepackage{amsmath, amssymb, amsthm, mathrsfs}
\usepackage{indentfirst}
\usepackage{graphicx, tikz, subfigure}
\usepackage{geometry}
\usepackage{longtable}
\usepackage{caption}
\usepackage{multirow}
\numberwithin{equation}{section}
\usepackage{xcolor}
\usepackage{ulem}
\usepackage{hyperref}
\allowdisplaybreaks[4]
\usepackage{cleveref}
\crefname{section}{Section}{Sections}
\crefname{figure}{Figure}{Figures}
\crefname{table}{Table}{Tables}
\crefname{equation}{}{}
\crefname{theorem}{Theorem}{Theorems}
\crefname{lemma}{Lemma}{Lemmas}
\crefname{remark}{Remark}{Remarks}
\crefname{problem}{Problem}{Subproblems}

\newtheorem{theorem}{Theorem}[section]
\newtheorem{problem}{Problem}[section]
\newtheorem{remark}{Remark}[section]
\newtheorem{lemma}{Lemma}[section]
\newtheorem{proposition}{Proposition}[section]

\theoremstyle{definition}

\newtheorem{example}{\noindent Example}[subsection]

\crefname{ip}{Co-inversion Problem}{ips}
\newtheoremstyle{MyThmStyle}
{}
{}
{}
{}
{\bfseries}
{}
{ }
{\thmname{#1\thmnumber{ #2\hspace{0.5em}}}\thmnote{(#3)}}
\theoremstyle{MyThmStyle}

\crefname{subisp}{inverse source problem}{ips}
\crefname{subiop}{inverse obstacle problem}{ips}
\graphicspath{{figurenew/}}
\definecolor{bananamania}{rgb}{0.98, 0.91, 0.71}

\begin{document}
	
\title{Recovering source location, polarization, and shape of obstacle from elastic scattering data}
		\author{
			Yan Chang\thanks{School of Mathematics, Harbin Institute of Technology, Harbin, China. {\it $21B312002@stu.hit.edu.cn$}},
			Yukun Guo\thanks{School of Mathematics, Harbin Institute of Technology, Harbin, China. {\it ykguo@hit.edu.cn} (Corresponding author)},
			Hongyu Liu\thanks{Department of Mathematics, City University of Hong Kong, Hong Kong SAR, China. {\it hongyliu@cityu.edu.hk}},
			\ and Deyue Zhang\thanks{School of Mathematics, Jilin University, Changchun, China, {\it dyzhang@jlu.edu.cn}}
		}
	\date{}
	\maketitle

\begin{abstract}
We consider an inverse elastic scattering problem of simultaneously reconstructing a rigid obstacle and the excitation sources using near-field measurements. A two-phase numerical method is proposed to achieve the co-inversion of multiple targets. In the first phase, we develop several indicator functionals to determine the source locations and the polarizations from the total field data, and then we manage to obtain the approximate scattered field. In this phase, only the inner products of the total field with the fundamental solutions are involved in the computation, and thus it is direct and computationally efficient. In the second phase, we propose an iteration method of Newton's type to reconstruct the shape of the obstacle from the approximate scattered field. Using the layer potential representations on an auxiliary curve inside the obstacle, the scattered field together with its derivative on each iteration surface can be easily derived. Theoretically, we establish the uniqueness of the co-inversion problem and analyze the indicating behavior of the sampling-type scheme. An explicit derivative is provided for the Newton-type method. Numerical results are presented to corroborate the effectiveness and efficiency of the proposed method.
\end{abstract}

\textbf{Keywords:} Co-inversion, inverse scattering, inverse source, elastic wave, Newton-type method, sampling.

\section{Introduction}

The identification of multiple targets of distinct nature from the scattering data has significant applications in various areas such as nondestructive testing, medical imaging, and geophysical exploration. In the scenario of the inverse problems for the wave equations, the excitation sources emit the signal actively while the obstacle serves the role to make a passive reaction to the imposed information. Hence, the source and the obstacle are typically viewed as two intrinsically distinct components in the scattering system. Due to varying practical desires, the reconstruction of either the source points or the obstacle has received enduring attention.
Depending on the reconstructed targets, the inverse source problems usually aim to recover the source from the radiated field, where there is no obstacle. Meanwhile, the inverse obstacle problems serve the purpose to identify the obstacle with a given incident field.

Typical numerical methods for the inverse elastic source problems include the recursive algorithm by Bao et.al \cite{BaoChenLi16}, the sampling-type method \cite{Wang22}, the full waveform inversion method \cite{SP}, and the fast Bayesian method for seismic source inversion \cite{Long}. Meanwhile, elastic wave scattering problems have received ever-increasing attention in recent years. For instance, Ji et. al \cite{JiLiuXiDSM} considered the inverse elastic scattering and proposed three direct sampling methods for location and shape reconstruction using different components of the far field patterns. Chen and Huang \cite{RTMChen} introduced the reverse time migration method to reconstruct the extended obstacle from the scattered field. Recently, the authors of \cite{LZ22} present a study on the time reversal method to recover multiple elastic particles in three dimensions.  In addition, classical algorithms for the inverse elastic scattering problems include the linear sampling method \cite{Alves, Arens}, the factorization method \cite{CKfactorization}, and recently the iterative method \cite{BaoHu} as well as the method of topological derivative by Guizina et. al \cite{Guizina}.

In many scenarios, both the obstacle and the source are unknown, which makes it meaningful to simultaneously reconstruct the two targets with the passive measurements, namely the measured wave data generated by the anomalous source. As a composition of the aforementioned problems, the co-inversion problem is more complicated and practically significant. Compared with the vast studies on single-inversion problems, studies on the co-inversion problem are relatively rare. For some recent works on co-inversion problems, we refer to \cite{era, LiLiuMa19, LiLiuMa21, ZGWC22, ZWG22}. In this paper, we consider an inverse elastic problem to simultaneously reconstruct the rigid obstacle and its excitation source points from time-harmonic total field data. The overall idea of the current work is divided into two steps. We first recover the source points together with the polarization from the total field data by the direct imaging method. Once the sources have been retrieved, the next step is to determine the shape of the obstacle by an easy-to-implement Newton-type iteration method.

 To be specific, we summarize the salient features of the proposed method as follows. First, we propose a two-phase sampling method to determine the source locations and the polarization from the total field. In the first phase, we implement the sampling procedure toward the spatial location, and the source location is identified via the significant maximizer of the indicator function. Then in the second phase, another sampling scheme is proposed to find the polarization direction. Second, by incorporating the idea of the traditional decomposition method \cite{CK19} into the Newton iteration, we develop a novel Newton-type framework by treating the ill-posedness and the nonlinearity of the inverse problem separately. In particular, by the Helmholtz decomposition and the layer potential techniques, the derivative of the boundary-to-data mapping can be calculated easily, thus the algorithm requires neither a forward solver nor alternative iterations between the sources and the obstacle involved in the novel Newton-type method. Hence, it is computationally efficient. Third, we demonstrate the effectiveness and efficiency of the proposed method through extensive numerical experiments. Furthermore, we extend this method to the co-inversion in the three-dimensional problem. Last but not least, the proposed hybrid method is comprised of a sampling scheme for source recovery and an iteration scheme for obstacle recovery, which are novel in their own right to solve the inverse source problem and the inverse obstacle scattering problem, respectively.

The rest of this paper is arranged as follows: In the next section, we introduce the co-inversion problem under consideration and address the uniqueness results. In \Cref{sec:DSM}, we propose two sampling schemes to identify the source locations and polarization from the measured total field. Mathematical justifications for the sampling schemes are provided. By subtracting the incident field due to the reconstructed sources from the total field, the co-inversion problem is reformulated into an inverse obstacle scattering problem. Then, a Newton-type method based on the single layer technique is proposed in \Cref{sec:Newton} to recover the shape of the obstacle from the approximate scattered field. In \Cref{sec:numerical}, we conduct several numerical experiments to verify the effectiveness and efficiency of our method. Finally, some conclusions are drawn in \Cref{sec:conclusion}.

\section{Problem setting and uniqueness}\label{sec:model}
In this section, we first give a brief description of the forward and inverse problems under consideration. Then a uniqueness result will be addressed.

\subsection{Model problem}
Let $D\subset\mathbb{R}^d\,(d=2,3)$ be an open and bounded Lipschitz domain such that the exterior $\mathbb{R}^d\backslash\overline{D}$ is connected. Assume that $\mathbb{R}^d\backslash\overline{D}$ is occupied by a homogeneous and isotropic elastic medium with a normalized  mass density. Let $\omega>0$ be the angular frequency and $\lambda,\,\mu$ be the Lam\'{e} constants such that $\mu>0,d\lambda+2\mu>0$. Denote by $k_p=\omega/\sqrt{\lambda+2\mu}=\omega/c_p$ and $k_s=\omega/\sqrt{\mu}=\omega/c_s$ respectively the compressional and shear wave numbers. Given a generic point $z\in\mathbb{R}^d\backslash\overline{D}$ and polarization $\bm{p}\in\mathbb{S}^{d-1}:=\{x\in\mathbb{R}^d:|x|=1\}$, the incident field $\bm{u}^i$ due to the source located at $z$ satisfies the Navier equation
\begin{align}\label{eq:ui}
\Delta^*\bm{u}^i+\omega^2\bm{u}^i=-\delta(x-z)\bm{p}\quad\text{in }\mathbb{R}^d\backslash\overline{D},
\end{align}
where $\delta(x-z)$ is the Dirac delta distribution at point $z,$ the Lam\'{e} operator $\Delta^*$ is defined by
\[
\Delta^*\bm{u}:=\mu\Delta\bm{u}+(\lambda+\mu)\nabla\nabla\cdot\bm{u}.
\]
Explicitly, we have
\begin{align}
\label{eq:inc}
\bm{u}^i=\bm{u}^i(x; z,\bm{p})=\mathbb{G}(x, z)\bm{p},\quad x\in\mathbb{R}^d\backslash(\overline{D}\cup\{z\}),
\end{align}
where $\mathbb{G}(x, z)$ is the fundamental solution to the Navier equation, i.e., (cf. \cite{BaoXuYin17})
\begin{equation}\label{eq:Green}
\mathbb{G}(x, z)=\frac{1}{\mu}\Phi_s(x, z)\mathbb{I}_d+\frac{1}{\omega^2}\nabla_x\nabla_x^\top(\Phi_s(x, z)-\Phi_p(x, z)).
\end{equation}
Here, $\mathbb{I}_d$ is the $d\times d$ identity matrix, $\Phi_\alpha (\alpha=p, s)$ denotes the fundamental solution to the Helmholtz equation with wave number $k_\alpha$, i.e.,
\begin{equation}\label{eq:Phi}
\Phi_\alpha(x, z)=\left\{
\begin{aligned}
&\frac{\mathrm{i}}{4}H_0^{(1)}(k_\alpha|x-z|),&&d=2,\\
&\frac{\mathrm{e}^{\mathrm{i}k_\alpha|x-z|}}{4\pi |x-z|},&&d=3,
\end{aligned}
\right.
\qquad \alpha=p, s,
\end{equation}
where $H_n^{(1)}$ the Hankel function of the first kind of order $n$. It holds that $\mathbb{G}=\mathbb{G}_p+\mathbb{G}_s$ where 
\begin{align}\label{eq:GpGs}
\mathbb{G}_p(x, z)=-\frac{1}{\mu k_s^2}\nabla_x\nabla_x^\top \Phi_p(x, z),\quad
\mathbb{G}_s(x, z)=\frac{1}{\mu}\left(\mathbb{I}_d+\frac{1}{k_s^2}\nabla_x\nabla_x^\top\right)\Phi_s(x, z),
\end{align}

Let the obstacle be illuminated by $\bm{u}^i,$ then the displacement field of the scattered wave is described by a solution $\bm{v}$ of the boundary value problem
\begin{equation}\label{eq:boundaryvalueproblem}
\left\{
\begin{aligned}
\Delta^*\bm{v}+\omega^2\bm{v}&=0,\quad&&\text{in }\mathbb{R}^d\backslash\overline{D},\\
\bm{u}&=0,&&\text{on }\partial D,
\end{aligned}
\right.
\end{equation}
where $\bm{u}=\bm{u}^i+\bm{v}$ is the total field with the scattered field $\bm{v}$ satisfying the Kupradze-Sommerfeld radiation condition
\begin{align*}
\lim\limits_{\rho\to\infty}\rho^{\frac{d-1}{2}}(\partial_\rho\bm{v}_p-\mathrm{i}k_p\bm{v}_p)=0,\quad\lim\limits_{\rho\to\infty}\rho^{\frac{d-1}{2}}(\partial_\rho\bm{v}_s-\mathrm{i}k_s\bm{v}_s)=0,\quad\rho=|x|.
\end{align*}
Here, $\bm{v}_p=-\frac{1}{k_p^2}\nabla\nabla\cdot\bm{v}$ is the compressional component of $\bm{v}.$ The shear component of $\bm{v}$ is given by
\begin{align*}
\bm{v}_s=
\left\{
\begin{aligned}
\frac{1}{k_s^2}\mathbf{curl}\text{curl}\bm{v},\quad d=2,\\
\frac{1}{k_s^2}\mathbf{curl}\mathbf{curl}\bm{v},\quad d=3,
\end{aligned}
\right.
\end{align*}
where the curl operators are defined by
\begin{align*}
&\text{curl}\bm{w}=\partial_{x_1}w_2-\partial_{x_2}w_1,\quad \mathbf{curl}w=(\partial_{x_2}w,-\partial_{x_1}w)^\top,\quad d=2,\\
&\mathbf{curl}\bm{w}=(\partial_{x_2}w_3-\partial_{x_3}w_2,\partial_{x_3}w_1-\partial_{x_1}w_3,\partial_{x_1}w_2-\partial_{x_2}w_1)^\top,\quad d=3.
\end{align*}
Here $w$ is a scalar function, and $\bm{w}=(w_1,w_2)^\top$ or $(w_1,w_2,w_3)^\top$ denotes a vector function.

For any solution $\bm{v}$ of equation \eqref{eq:boundaryvalueproblem}, the Helmholtz decomposition splits it into its compressional and shear parts:
\begin{equation}\label{eq:split}
\bm{v}=
\begin{cases}
\nabla\phi+\bf{curl}\psi, & d=2,\\
\nabla\phi+\bf{curl}\bm{\psi}, & d=3,
\end{cases}
\end{equation}
where $\phi$ and $\psi$ are scalar potential functions and $\bm{\psi}$ is a vector potential function fulfilling $\nabla\cdot\bm{\psi}=0$. 

Combining \eqref{eq:boundaryvalueproblem} and \eqref{eq:split} gives the Helmholtz equations:
\begin{align}\label{eq:Helmholtz2D}
&\Delta \phi+k_p^2\phi=0,\quad\Delta\psi+k_s^2\psi=0,\ \ d=2,\\\label{eq:Helmholtz3D}
&\Delta\phi+k_p^2\phi=0,\quad\Delta\bm{\psi}+k_s^2\bm{\psi}=\bm{0},\ \ d=3.
\end{align}
In addition, $\phi$, $\psi$ and $\bm{\psi}$ are supposed to satisfy the Sommerfeld radiation condition
\begin{align}\label{eq:radiation}
\begin{aligned}
&\lim\limits_{\rho\to\infty}\rho^{\frac{d-1}{2}}(\partial_\rho\phi-\mathrm{i}k_p\phi)=0,\quad
\lim\limits_{\rho\to\infty}\sqrt{\rho}(\partial_\rho\psi-\mathrm{i}k_s\psi)=0,\\
&\lim\limits_{\rho\to\infty}\rho(\mathbf{curl}\bm{\psi}\times\hat{x}-\mathrm{i}k_s\bm{\psi})=0,
\quad\rho=|x|.
\end{aligned}
\end{align}

It has been proven in \cite{Bramble08} that there exists a unique solution $\bm{v}\in \left(H_\text{loc}^1(\mathbb{R}^d\backslash\overline{D})\right)^d$ to the direct problem \eqref{eq:boundaryvalueproblem} and \eqref{eq:radiation}.

In this paper, we take $B_R:=\{x\in\mathbb{R}^d:|x|<R\}$ containing $D$ such that $B_R\backslash\overline{D}$ is connected. For $N\in\mathbb{N}_+$, let $S:=\cup_{j=1}^N\{z_j\}\subset B_R\backslash\overline{D}$ be a set of $N$ distinct source points and $P:=\cup_{j=1}^N\{\bm{p}_j\}\subset\mathbb{S}^{d-1}$ be the set of polarization directions.
Given the incident field $\bm{u}^i(x; z_j, \bm{p}_j),\,j=1,\cdots, N,$ of the form \eqref{eq:inc}, we collect the total field $\bm{u}(x; z_j,\bm{p}_j)=\bm{u}^i(x; z_j, \bm{p}_j)+\bm{v}(x; z_j, \bm{p}_j)$ on the measurement curve $\Gamma_R:=\partial B_R=\{x\in\mathbb{R}^d:|x|=R\},$ where $\bm{v}(x; z_j, \bm{p}_j)$ is the scattered field corresponding to the incident field $\bm{u}^i(x; z_j,\bm{p}_j).$
Then, the co-inversion problem we are interested in is stated as:

\begin{problem}[Co-inversion problem]\label{prob:coinversion}
Find the obstacle $\partial D$, source points $S$ and polarization directions $P$ simultaneously from the measurements $\mathbb{U}:=\{\bm{u}(x; z,\bm{p}):x\in\Gamma_R, z\in S, \bm{p}\in P\},$ i.e.,
	\begin{equation}\label{eq:inverseSetup}
	\mathbb{U}\to(\partial D, S, P).
	\end{equation}
\end{problem}
For an illustration of the geometry setup of \cref{prob:coinversion}, we refer to \cref{fig:Setup}.
\begin{figure}
	\centering
	\begin{tikzpicture}[thick]
	\draw [gray] circle (3cm);
	\draw node at (2.4, 2.4) {$\Gamma_R$};
	
	\pgfmathsetseed{2}	
	\clip (0, 0) circle (2.7cm);
	\foreach \p in {1,...,15}
	{\fill [red] (3*rand, 3*rand) circle (0.06);
	}
	
	\pgfmathsetseed{8}			
	\draw plot [smooth cycle, samples=5, domain={1:6}] (\x*360/8+2*rnd:0.4cm+1.65cm*rnd) node at (0, -.5) {$D$}[fill=bananamania];
	\draw node at (2,.7) {$\bm{u}^i$};
	\draw node at (1.5, 1.5) {$B_R$};
	\draw [blue, domain=150:210] plot ({2.4+.5*cos(\x)}, {0.5*sin(\x)});
	\draw [blue, domain=150:210] plot ({2.2+.6*cos(\x)}, {0.7*sin(\x)});
	\draw [->] (2.2,0)--(1.4,0);
	\draw [->] (-1.5,.5)--(-2.2,.65);
	\draw [blue, thick, domain=140:200] plot ({-1.3+.5*cos(\x)}, {.5+0.5*sin(\x)});
	\draw [blue, thick, domain=140:200] plot ({-1.4+.6*cos(\x)}, {.5+0.6*sin(\x)});
	\draw node at (-2,1) {$\bm{u}$};
	
	\fill[red] (2.2,0) circle(0.06);
	\end{tikzpicture}
	\caption{Illustration of the co-inversion for imaging the obstacle and sources. }\label{fig:Setup}
\end{figure}
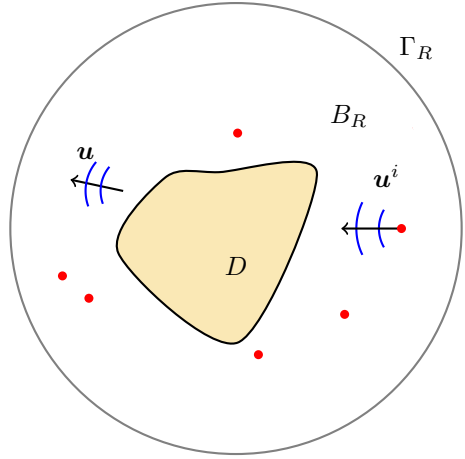

\subsection{Uniqueness}
In this subsection, we consider the uniqueness issue concerning \cref{prob:coinversion}. Specifically, we show that $S$ and $\partial D$ can be uniquely determined from the total-field measurements. We also refer to \cite{DLW20, DLW21, DLS21, HLL14} and the references therein for more studies on the uniqueness of inverse elastic scattering problems.

\begin{theorem}\label{thm:uniqueness}
	The source points $S$ can be uniquely determined by the total field $\mathbb{U}.$ 
 Let $N_0$ be defined by \eqref{eq:N}. If $N\ge N_0+1$ for one fixed angular frequency $\omega$ and one fixed polarization $\bm{p}\in\mathbb{S}^{d-1},$  then the obstacle $D$ can also be uniquely determined by the total field data $\mathbb{U}.$
\end{theorem}
\begin{proof}
	We first prove the unique identification of the source points by contradiction.
	
	Let $D_1$ and $D_2$ be two elastically rigid obstacles such that $D_1\cup D_2\subset B_R.$ Assume $w_1\ne w_2$
	be two different source points in $S$ and denote the total fields due to $(D_1,w_1)$ and $(D_2,w_2)$ by $\bm{u}(x;D_1,w_1),$ $\bm{u}(x;D_2,w_2),$ respectively.
	Assume that
	\begin{align}\label{eq:GammaR}
	\bm{u}(x;D_1,w_1)=\bm{u}(x;D_2,w_2),\quad \forall\, x\in\Gamma_R.
	\end{align}
    From the uniqueness of the exterior Dirichlet boundary value problem for the Navier equation, we derive that
    \[
	\bm{u}(x;D_1,w_1)=\bm{u}(x;D_2,w_2),\quad \text{in }\mathbb{R}^d\backslash\overline{B_R}.
	\]
     Further, the analyticity leads to the following fact
	\[
	\bm{u}(x;D_1,w_1)=\bm{u}(x;D_2,w_2),\quad \text{in }\mathbb{R}^d\backslash\left(\overline{D_1\cup D_2}\cup\{w_1\}\cup\{w_2\}\right).
	\]
	Let $x\to w_1,$ then $\bm{u}(x;D_1,w_1)$ tends to infinity. Meanwhile, from the fact $\bm{v}(x;D_\ell,w_\ell),(\ell=1,2)$ is bounded, we know that $\bm{u}(x;D_2,w_2)$ is bounded,
	 which leads to a contradiction. Thus, $w_1=w_2.$
	
In what follows, we prove by a contradiction argument that if $N\ge N_0+1,$ then $\partial D$ can be uniquely determined by $\mathbb{U}.$  

Assume that $D_1$ and $D_2$ are two bounded domains and $\bm{v}_i,\,i=1,2$ satisfy \eqref{eq:boundaryvalueproblem} with $D$ replaced by $D_i,$ respectively. Let $G$ be the unbounded component of the complement of $D_1\cup D_2$ and the total wave vanishes on $\partial G.$ Without loss of generality, we assume that $D^*:=(\mathbb{R}^d\backslash G)\backslash\overline{D}_2\neq\emptyset.$ Then $\bm{v}_2(\cdot, z)$ satisfies
\begin{align*}
\left\{
\begin{aligned}
&\Delta^* \bm{v}_2+\omega^2 \bm{v}_2=0,&&\text{in } D^*,\\\
&\bm{v}_2(\cdot, z)=-\bm{u}^i(\cdot, z),&&\text{on }\partial D^*.
\end{aligned}
\right.
\end{align*}
Let $\bm{w}(\cdot, z)=\bm{v}_2(\cdot, z)+\bm{u}^i(\cdot,z)$ for a fixed $z\in S,$ then $\bm{w}$ satisfies
\begin{align}
\left\{
\begin{aligned}
&\Delta^*\bm{w}+\omega^2\bm{w}=0,&&\text{in }D^*,\\
&\bm{w}=0,&&\text{on }\partial D^*.
\end{aligned}
\right.
\end{align}
As a result, $\bm{w}$ is a Dirichlet eigenfunction for $-\Delta^*$ in $D^*$ with $\omega^2$ the eigenvalue. Next, we show that the eigenfunctions $\bm{w}(\cdot,{z}_j),\,z_j\in S, j=1,\cdots, N_0+1$ corresponding to the same eigenvalue $\omega^2$ are linearly independent.
Assume that
\begin{align}\label{eq:dependent}
\sum_{j=1}^{N_0+1}c_j\bm{w}(x,{z}_j)=0,\quad x\in D^*,
\end{align}
holds for some constants $c_j$ and $N_0+1$ distinct source points $z_j, j=1,\cdots, N_0+1.$ Then by analyticity, \eqref{eq:dependent} is also satisfied in the exterior of some circle containing $D_1$ and $D_2.$ For a fixed $j_0\in[1,N_0+1],$ we take $h>0$ sufficiently small such that $x_s^{j_0}=z_{j_0}+\frac{h}{s}\nu(z_{j_0}),s=1,2,\cdots$ are in a neighborhood of $z_{j_0}.$ Then,
\[
c_{j_0}\bm{w}(x_s^{j_0},{z}_{j_0})=-\sum_{\substack{j=1,j\ne j_0}}^{N_0+1}c_j\bm{w}(x_s^{j_0},{z}_j).
\]
Further, we derive that
\begin{align}\label{eq:c0}
c_{j_0}\bm{u}^i(x_s^{j_0},{z}_{j_0})=-\sum_{\substack{j=1,j\ne j_0}}^{N_0+1}c_j\bm{u}^i(x_s^{j_0},{z}_j)-\sum_{j=1}^{N_0+1}c_j\bm{v}(x_s^{j_0},z_j).
\end{align}
Noticing the fact that $\bm{u}^i(x_s^{j_0},{z}_{j_0})$ becomes unbounded and the right hand side in \eqref{eq:c0} remains bounded while $s\to\infty,$ we derive that $c_{j_0}=0$ for $j_0=1,2,\cdots, N_0+1,$ which implies that $\bm{w}(\cdot,{z}_j),\,j=1,\cdots, N_0+1,$ are linearly independent.

We can proceed with the proof in the same way as in the proof of Theorem 5.2 in \cite{CK19}. Let $0<\lambda_1\le\lambda_2\le\cdots\le\lambda_m=\omega^2$ be the Dirichlet eigenvalues of $D^*$ that are smaller than or equal to $\omega^2$ and $\mu_1\le\mu_2\le\cdots\mu_m$ are the first $m$ eigenvalues of $B_R,$ then $\mu_m<\lambda_m=\omega^2.$ Based on the strong monotonicity property for the Dirichlet eigenvalues of $-\Delta^*,$ we obtain that the multiplicity $M$ of $\lambda_m$ is less than or equal to the sum of multiplicities of eigenvalues for the disk $B_R$ which are less than $\omega^2.$ In other words, $M\le N_0,$ which  contradicts with the fact that $N_0+1$ distinct incident waves yield $N_0+1$ linearly independent eigenfunctions with eigenvalue $\omega^2$ for $D^*$. Hence, $D_1=D_2.$
\end{proof}


\section{Recovering the source}\label{sec:DSM}
In this section, we develop several novel indicator functionals to determine the source points $S$ and the polarization directions $P$ from the total field $\mathbb{U}$. It deserves noting that, different from the existing direct sampling methods for the inverse source problem, we adopt the total field instead of the incident field in the imaging function. For convenience, we use $(\cdot,\cdot)$ for the real inner product on $\mathbb{C}^2$ and the overbar for the complex conjugate, and use $\langle\cdot,\cdot\rangle$ for the inner product defined on $(L^2(\Gamma_R))^2$.

\subsection{Recovering the location}\label{subsec:location}
The aim of this subsection is to determine the locations of source points from $\mathbb{U}.$
Let $\Omega\subset\mathbb{R}^d$ be a bounded sampling domain such that $D\cup S\subset\Omega.$ For any sampling point $y\in\Omega, $ we define $N$ indicator functionals as follows:
\begin{equation}\label{eq:I}
I_j^{\bm{q}}(y)=\left\langle\bm{u}(\cdot; z_j), \mathbb{G}(\cdot, y)\bm{q}\right\rangle,\quad j=1,\cdots, N,
\end{equation}
where $\bm{q}\in\mathbb{S}^{d-1}.$ We would like to point out that, to determine the source points $S,$ the auxiliary polarization $\bm{q}$ is not necessarily the same as $\bm{p}_j.$

To investigate the characteristics of the indicator functions \eqref{eq:I}, several crucial lemmas are needed.
\begin{lemma}{$\cite[Lemma~3.4]{RTMChen}$}\label{lem:GG}
	For all $y, z\in B_R,$ $y\ne z,$ it holds that
	\begin{align*}
	\omega c_\alpha \left\langle\mathbb{G}_\alpha(x, z),\mathbb{G}_\alpha(x, y)\right\rangle&=\Im\{\mathbb{G}_\alpha(y, z)\}+\mathbb{W}_\alpha^r(y, z),\quad \alpha=p, s,\\
	\omega\left\langle\mathbb{G}_p(x, z), \mathbb{G}_s(x, y) \right\rangle & = \mathbb{W}_{ps}^r(y, z) ,\\
	\omega\left\langle\mathbb{G}_s(x, z), \mathbb{G}_p(x, y) \right\rangle & =\mathbb{W}_{sp}^r(y, z) ,
	\end{align*}
	where $\|\mathbb{W}_\alpha^r\|_{L^\infty(B_R\times B_R)}+\|\nabla_x\mathbb{W}_\alpha^r\|_{L^\infty(B_R\times B_R)}\le CR^{-\frac{d-1}{2}}$ holds uniformly with $\alpha\in\{p, s, ps, sp\}.$ Here,  $\|A\|_{L^\infty(B_R\times B_R)}=\max\limits_{i, j=1,\cdots, d}\|A_{ij}\|_{L^\infty(B_R\times B_R)}$ for  $A(x, y)=(A_{ij}(x, y))\in\mathbb{C}^{d\times d} (i, j=1,\cdots, d).$  
\end{lemma}
 
For the trace of Green tensors, we have the following property.
\begin{lemma}\label{lem:trace}
For $d=2,3$,  it holds that
\begin{equation}\label{eq:trace}
{\rm tr}(\mathbb{G}_\alpha)=C_\alpha\Phi_\alpha, \quad \alpha=p, s,
\end{equation}
where ${\rm tr}$ denotes the sum of the diagonal terms and
$$
C_\alpha=
\begin{cases}
\dfrac{1}{{\lambda+2\mu}}, & \alpha=p,\vspace{2mm} \\ 
\dfrac{d-1}{\mu}, & \alpha=s.
\end{cases}
$$
\end{lemma}

\begin{proof}
Define
$$
f_n=
\begin{cases}
H_n^{(1)}, & d=2,\\
h_n^{(1)}, & d=3,
\end{cases}
\quad n=0,1,2,\cdots,
$$
and
$$
\sigma_\alpha=
\begin{cases}
\pi, & d=2,\\
k_\alpha, & d=3,
\end{cases}
\quad \alpha=p, s.
$$
Then we see that
\begin{align}
\label{eq:fn} f'_n(z) & =\frac{n}{z}f_n(z)-f_{n+1}(z),\quad n=0,1,2,\cdots,\\
\label{eq:f1} d f_1(z)/{z}&=f_0(z)+f_2(z),\quad d=2,3,\\
\label{eq:phi_alpha}\Phi_\alpha(x, y) & =
\frac{\mathrm{i}\sigma_\alpha}{4\pi}f_0(k_\alpha|x-y|),\quad \alpha=p, s.
\end{align}

By \eqref{eq:fn}, it can be straightforwardly derived that
$$
\nabla_x\nabla_x^\top f_0(k_\alpha|x-y|)=
k_\alpha^2f_2(k_\alpha|x-y|)\frac{(x-y)\otimes(x-y)}{|x-y|^2}-\frac{k_\alpha}{|x-y|}f_1(k_\alpha|x-y|)\mathbb{I}, \quad \alpha=p, s,
$$
where $\otimes$ denotes the outer product. Hence, together with \eqref{eq:GpGs} and \eqref{eq:phi_alpha}, the Green tensors can be rewritten as
\begin{align*}
&\mathbb{G}_p(x, y)=-\frac{\mathrm{i}\sigma_p}{4\pi c_p^2}\left(
f_2(k_p|x-y|)\frac{(x-y)\otimes(x-y)}{|x-y|^2}-\frac{f_1(k_p|x-y|)}{k_p|x-y|}\mathbb{I}
\right),\\
&\mathbb{G}_s(x, y)=
\frac{\mathrm{i}\sigma_s}{4\pi c_s^2}\left[
\left(f_0(k_s|x-y|)-\frac{f_1(k_s|x-y|)}{k_s|x-y|}\right)\mathbb{I}+f_2(k_s|x-y|)\frac{(x-y)\otimes(x-y)}{|x-y|^2}
\right].
\end{align*}
Further, using \eqref{eq:f1}, we obtain that
\begin{align*}
{\rm tr}\left(\mathbb{G}_p(x, y)\right)&=-\frac{\mathrm{i}\sigma_p}{4\pi c_p^2}\left(f_2(k_p|x-y|)-\frac{df_1(k_p|x-y|)}{k_p|x-y|}\right)
=\frac{\mathrm{i}\sigma_p}{4\pi c_p^2}f_0(k_p|x-y|)
=\frac{1}{c_p^2}\Phi_p(x, y),\\
{\rm tr}\left(\mathbb{G}_s(x, y)\right)&=\frac{\mathrm{i}\sigma_s}{4\pi c_s^2}\left(df_0(k_s|x-y|)-\frac{df_1(k_s|x-y|)}{k_s|x-y|}+f_2(k_s|x-y|)\right)\\
&=\frac{\mathrm{i}\sigma_s(d-1)}{4\pi c_s^2}f_0(k_s|x-y|)\\
&=\frac{d-1}{c_s^2}\Phi_s(x, y),
\end{align*}
which completes the proof. 
\end{proof}

To analyze the indicating behavior of the indicator \eqref{eq:I}, we rewrite \eqref{eq:I} into two parts as follows:
$$
I_j^{\bm{q}}(y)=I_j^{S,\bm{q}}(y)+I_j^{D,\bm{q}}(y),
$$
where 
\begin{align}
\label{eq:IjS}I_j^{S,\bm{q}}(y)&=\left\langle\bm{u}^i(\cdot, z_j),\mathbb{G}(\cdot, y)\bm{q}\right\rangle,\\
\label{eq:IjD}I_j^{D,\bm{q}}(y)&=\left\langle\bm{v}(\cdot, z_j),\mathbb{G}(\cdot, y)\bm{q}\right\rangle.
\end{align}

\begin{proposition}
	Let $I_j^{S,\bm{q}}(y), j=1,\cdots, N,$ be defined by \eqref{eq:IjS}. Then we have
	\begin{align*}
	I_j^{S,\bm{q}}(y)=\sum_{\alpha\in\{p, s\}}\frac{1}{\omega c_\alpha}\Im\big\{\big(\bm{p}_j,\mathbb{G}_\alpha(y, z_j)\bm{q}\big)\big\}+\mathcal{O}\left(R^{-\frac{d-1}{2}}\right),\quad R\to\infty.
	\end{align*}
\end{proposition}

\begin{proof}
	To see this property, we find that $j=1,\cdots, N,$ and $y\in\Omega,$
	\begin{align*}
	&\left\langle\mathbb{G}(\cdot, z_j)\bm{p}_j,\mathbb{G}(\cdot, y)\bm{q}\right\rangle\\
	=&\sum_{\alpha\in\{p, s\}} \left\langle\mathbb{G}_\alpha(x, z_j)\bm{p}_j, \mathbb{G}_\alpha(x, y)\bm{q}\right\rangle+\left\langle\mathbb{G}_p(x, z_j)\bm{p}_j, \mathbb{G}_s(x, y)\bm{q}\right\rangle+\big\langle\mathbb{G}_s(x, z_j)\bm{p}_j, \mathbb{G}_p(x, y)\bm{q}\big\rangle \\
	= & \sum_{\alpha\in\{p, s\}} \frac{1}{\omega c_\alpha}\Im\big\{\big(\bm{p}_j,\mathbb{G}_\alpha(y, z_j)\bm{q}\big)\big\}+\frac{1}{\omega}\left(\frac{(\bm{p}_j,\mathbb{W}_p^r\bm{q})}{c_p}+\frac{(\bm{p}_j,\mathbb{W}_s^r\bm{q})}{c_s}+(\bm{p}_j,\mathbb{W}_{ps}^r\bm{q})+(\bm{p}_j,\mathbb{W}_{s p}^r\bm{q})\right) \\
	=&\sum_{\alpha\in\{p, s\}}\frac{1}{\omega c_\alpha}\Im\big\{\big(\bm{p}_j,\mathbb{G}_\alpha(y, z_j)\bm{q}\big)\big\}+\mathcal{O}\left(R^{-\frac{d-1}{2}}\right),
	\end{align*}
	which completes our proof.
\end{proof} 

Let $J_0$ be the Bessel function of order zero. From \Cref{lem:trace}, we know that the crucial quantity
\begin{align*}
\frac{1}{\omega c_\alpha}\Im\{{\rm tr}(\mathbb{G}_\alpha(y, z_j))\}=\left\{
\begin{aligned}
&\frac{C_\alpha}{4\omega c_\alpha}J_0(k_\alpha|y-z_j|),&d=2,\\
&\frac{C_\alpha}{{4\pi}\omega c_\alpha}\frac{\sin(k_\alpha|y-z_j|)}{|y-z_j|},&d=3,
\end{aligned}
\right.
\end{align*}
obtains its significant value at $y=z_j, j=1,\cdots, N.$ Otherwise, this term is relatively small, which implies that the indicator functional proposed in \eqref{eq:IjS} can indicate the presence of source points.

To analyze the indicating behaviors of \eqref{eq:IjD}, we notice that through the single-layer representation, the  scattered field can be given by
\begin{align}\label{eq:scattered}
	\bm{v}_j(x)=(\mathcal{S}_1\bm{\varphi}_j)(x)=\int_{\partial D}\mathbb{G}(x, w)\bm{\varphi}_j(w)\mathrm{d}s(w),
\end{align}
with $\bm{\varphi}_j\in(L^2(\Lambda))^d$ the density corresponding to the $j$-th source point $z_j.$
	
\begin{proposition}
	Let $I_j^{D,\bm{q}}, j=1,2,\cdots, N,$ be defined by \eqref{eq:IjD}, then it holds that:
	\begin{equation}
	I_j^{D,\bm{q}}(y)=\int_{\partial D}\bm{\varphi}_j(w)\cdot\bm{q}\sum_{\alpha\in\{p, s\}}\frac{1}{\omega c_\alpha}\Im(\mathbb{G}_\alpha(w, y))\mathrm{d}s(w)+\mathcal{O}\left(R^{-\frac{d-1}{2}}\right).
	\end{equation}
\end{proposition}

\begin{proof}
       Substituting \eqref{eq:scattered} into \eqref{eq:IjD} derives that
	$$
	I_j^{D,\bm{q}}(y)=\left\langle\bm{v}_j(x), \mathbb{G}(x, y)\bm{q}\right\rangle
	=\int_{\partial D}\bm{\varphi}_j(w)\cdot\bm{q}\mathrm{d}s(w)\left\langle\mathbb{G}(x, w), \mathbb{G}(x, y)\right\rangle.
	$$
	Noticing \Cref{lem:GG}, we obtain that
	\begin{align*}
	\left\langle\mathbb{G}(x, w), \mathbb{G}(x, y)\right\rangle
	& = \left\langle\mathbb{G}_p(x, w)+\mathbb{G}_s(x, w), \mathbb{G}_p(x, y)+\mathbb{G}_s(x, y)\right\rangle \\
	& = \left\langle\mathbb{G}_p(x, w), \mathbb{G}_p(x, y)\right\rangle
	     +\left\langle\mathbb{G}_s(x, w), \mathbb{G}_s(x, y)\right\rangle \\
	&\quad +\left\langle\mathbb{G}_p(x, w), \mathbb{G}_s(x, y)\right\rangle
	+\left\langle\mathbb{G}_s(x, w), \mathbb{G}_p(x, y)\right\rangle \\
	& =\frac{1}{\omega c_p}\Im\{\mathbb{G}_p(w, y)\}+\frac{1}{\omega c_s}\Im\{\mathbb{G}_s(w, y)\}+\mathbb{W}(w, y),
	\end{align*}
	where 
	$$
	\mathbb{W}=\frac{1}{\omega}\left(\frac{\mathbb{W}_p^r}{c_p}+\frac{\mathbb{W}_s^r}{c_s}+\mathbb{W}_{ps}^r+\mathbb{W}_{sp}^r\right).
	$$
	Further, 
	\begin{align*}
	I_j^{D,\bm{q}}(y)&=\int_{\partial D}\bm{\varphi}_j(w)\cdot\bm{q}\left(\sum_{\alpha\in\{p, s\}}\frac{1}{\omega c_\alpha}\Im(\mathbb{G}_\alpha(w, y))+\mathbb{W}(w, y)\right)\mathrm{d}s(w)\\
	&=\int_{\partial D}\bm{\varphi}_j(w)\cdot\bm{q}\sum_{\alpha\in\{p, s\}}\frac{1}{\omega c_\alpha}\Im(\mathbb{G}_\alpha(w, y))\mathrm{d}s(w)+\mathcal{O}\left(R^{-\frac{d-1}{2}}\right),
	\end{align*}
	which completes the proof.
	\end{proof}
	From \Cref{lem:trace}, we know that $I^{D,\bm{q}}_j(y)$ attains its maximum at $y=w\in\partial D,$ which indicates the presence of the obstacle.
	
Combining the above analysis, we know that by choosing proper polarization $\bm{q}$, the indicator functional \Cref{eq:I} can indicate the presence of the source points and the obstacle. As will be seen in later numerical experiments, the indicator functional in \Cref{eq:I} identifies the source points accurately, while the obstacle can not be recognized properly since only a single source point is involved. 

\subsection{Identifying the polarization directions}\label{subsec:polarization}
In \Cref{subsec:location}, we point out that the choice of the polarization $\bm{q}$ in the indicating function $I_j^{\bm{q}}$ need not be the same as $\bm{p}_j.$ Nevertheless, it is the source location $z_j$ and the polarization $\bm{p}_j$ that determine the incident field $\bm{u}^i(\cdot; z_j,\bm{p}_j)$ collectively. In other words, only the source location can not determine the incident field. Therefore, we also need to determine the polarizations $\bm{p}_j$ from $\mathbb{U}$.

For convenience, we only consider the case for $d=2$. Let $[0,\pi)$ be an angular sampling interval. Choose an angle $\theta_1\in[0,\frac{\pi}{2})$ randomly and let $\theta_2=\pi-\theta_1.$ Compute
\begin{align}
\label{eq:Ii}
I_j^{\bm{q}_i}(y)=\left\langle\bm{u}(\cdot; z_j), \mathbb{G}(\cdot, y)\bm{q}_i\right\rangle,\quad j=1,\cdots, N,
\end{align}
with the polarization $\boldsymbol{q}_i=(\cos\theta_i,\sin\theta_i),\,i=1,2,$ and collect all the maximizer $\tilde{z}_j^i$ of each indicator function $|I_j^{\bm{q}_i}(y)|,\,i=1,2,j=1,2,\cdots, N,$ which can be viewed as the reconstructed source points.

Now, for each $j=1,2,\cdots, N,$ we have obtained two reconstructed source points $\tilde{z}_j^i,\,i=1,2.$ Define a uniform partition for $[0,\pi)$ by $\theta_\ell=\frac{\ell\pi}{N_q},\,\ell=1,2,\cdots, N_q.$ Then for each $\tilde{z}_j^i,\,i=1,2,$ we compute the indicator functions defined by 
\begin{align}
\label{eq:Iq}
I_j^i(\boldsymbol{q}_\ell)=\left\langle\bm{u}(\cdot; z_j), \mathbb{G}(\cdot, \tilde{z}_j^i)\bm{q}_\ell\right\rangle,\quad j=1,\cdots, N,
\end{align}
with $\bm{q}_\ell=(\cos\theta_\ell,\sin\theta_\ell)^\top.$
For each $i=1,2,$ we take $\ell_j^i$ such that 
$$\left|I_j^{i}(\boldsymbol{q}_{\ell_j^i})\right|=\max_{\ell=1,\cdots,N_q}\left|I_j^{i}(\boldsymbol{q}_\ell)\right|.$$
Further, we take $i_0$ such that
$$\left|I_j^{i_0}(\bm{q}_{\ell_j^{i_0}})\right|=\max_{i=1,2}\left|I_j^{i}(\boldsymbol{q}_{\ell_j^i})\right|,$$
and take $\tilde{\boldsymbol{q}}_{j}=\boldsymbol{q}_{\ell_j^{i_0}}$ as an approximate to the polarization $\boldsymbol{p}_j.$

Once the polarization $\bm{p}_j,\,j=1,2,\cdots, N$ is approximated by $\tilde{\bm{q}}_{j}$, we shall take $\tilde{z}_j$ such that 
\begin{align}\label{eq:Izj}
\left|I_j^{\tilde{\bm{q}}_{j}}(\tilde{z}_j)\right|=\max_{i=1,2}\left|I_j^{\tilde{\bm{q}}_{j}}(\tilde{z}_j^{i})\right|
\end{align}
as the reconstruction to the exact source points $z_j.$

Finally, the inversion scheme for recovering the source is summarized in Algorithm 1.
\begin{table}\label{algorithm:source}
	\begin{tabular}{cp{.8\textwidth}}
		\toprule
		\multicolumn{2}{l}{{\bf Algorithm 1:}\quad Determine the source points and polarization from the total field.} \\\hline
		{\bf Step 1} & {\bf Data collection:} Measure the data $\mathbb{U}=\{\bm{u}(x; z_j, \bm{p}_j): x\in\Gamma_R,  j=1,\cdots, N\}$;\\
		{\bf Step 2} & {\bf Determine the location:} \\
		& {\bf (a)} Select a sampling domain $\Omega\subset B_R$ such that $D\cup S\subset\Omega$ and generate the sampling grid $\mathcal{T}$ over the sampling domain $\Omega.$ Choose two polarizations $\bm{q}_i=(\cos\theta_i,\sin\theta_i)$ with $\theta_1\in\Big[0,\frac{\pi}{2}\Big),\,\theta_2=\pi-\theta_1;$\\
				 & {\bf (b)} For each sampling point $y\in\mathcal{T},$ compute $I_j^{\bm{q}_i}(y)$, $i=1,2,$ $j=1,2,\cdots, N;$\\
				&{\bf (c)} Collect the maximizer  $\tilde{z}_j^i$ of each indicator $|I_j^{\bm{q}_i}(y)|,\,i=1,2,\,j=1,2,\cdots, N;$ \\
		{\bf Step 3} & {\bf Determine the polarization:} \\
				& {\bf (a)} Define the sampling polarizations $\bm{q}_\ell=\left(\cos \frac{\ell\pi}{N_q},\sin \frac{\ell\pi}{N_q}\right)^\top$,  $\,\ell=1,2,\cdots, N_q$;\\
				&{\bf (b)} For each $j=1,2,\cdots, N,$ and $\ell=1,2,\cdots, N_q,$ compute the indicator function  \eqref{eq:Iq} and take $\ell_j^i$ such that $I_j^i(\bm{q}_{\ell_j^i})=\max\limits_{\ell=1,2,\cdots, N_q}|I_j^i(\bm{q}_{\ell})|$; \\
				&{\bf (c)} For each $j=1,2,\cdots, N$, we take $i_0$ such that
				$$I_j^{i_0}(\bm{q}_{\ell_j^{i_0}})=\max_{i=1,2}\left|I_j^{i}(\boldsymbol{q}_{\ell_j^i})\right|,$$
				and take $\tilde{\bm{q}}_{j}=\bm{q}_{\ell_j^{i_0}}$ as an approximation to the polarization $\boldsymbol{p}_j.$
				\\
		{\bf Step 4}& Once the polarization $\boldsymbol{p}_j$ is approximated by $\tilde{\boldsymbol{q}}_j$, we further take $\tilde{z}_j$ defined by \eqref{eq:Izj} as the reconstruction to the source point $z_j$;\\
		\bottomrule
	\end{tabular}
\end{table}


\section{Recovering the obstacle}\label{sec:Newton}

The identification of sources in the previous section enables us to convert the co-inversion problem into an inverse obstacle scattering problem by subtracting the incident wave from the total field. In this section, we shall further determine the obstacle from the approximate scattered field.

Due to the existence of unknown sources, the scattered field can not be measured directly. Nevertheless, from \Cref{sec:DSM}, the source points and the polarizations can be recovered from the total field. Denote by $\tilde{z}_j$ and $\tilde{\bm{p}}_j, j=1,\cdots, N$ the reconstructed source points and polarization directions. Then the scattered field corresponding to $z_j$ can be approximated by subtracting the incident field $\bm{u}^i(x;\tilde{z}_j, \tilde{\bm{p}}_j)$ due to the numerical source point $\tilde{z}_j$ from the measured total field $\bm{u}(x; z_j, \bm{p}_j),$ i.e.,
\begin{align}\label{eq:appro_us}
\tilde{\bm{v}}(x; z_j,\bm{p}_j)=\bm{u}(x; z_j,\bm{p}_j)-\bm{u}^i(x;\tilde{z}_j,\tilde{\bm{p}}_j), \quad j=1,\cdots,N.
\end{align}
Then the inverse problem is simplified to the inverse scattering problem: reconstruct $\partial D$ from $$\{\tilde{\bm{v}}(x; z_j,\bm{p}_j):x\in\Gamma_R, j=1,\cdots, N\}.$$

For ease of exposition, we mainly consider the reconstruction of the obstacle with a single (exact or approximate) incident wave.  In \Cref{subsec1}, we shall propose a novel Newton-type method for the conventional inverse obstacle scattering problem. Without loss of generality, we shall design the Newton-type method for the more general inverse elastic scattering problem and will denote the scattered field corresponding to $z_j$ as $\bm{v}(x; z_j,\bm{p}_j).$ The co-inversion problem can be tackled by substituting the scattered field $\bm{v}(x; z_j,\bm{p}_j)$ with the approximate scattered field $\tilde{\bm{v}}(x; z_j,\bm{p}_j).$ In what follows, we always 
assume that, under certain prior information about the obstacle, we can choose a closed surface $\Lambda\subset D$ such that $\omega^2$ is not the Dirichlet eigenvalue for $-\Delta^*$ inside $\Lambda$.

\subsection{Layer potentials for approximating the wave field}\label{subsec1}
To establish the iteration scheme, we first represent the approximate scattered field as an appropriate layer potential defined on the auxiliary surface $\Lambda$. Since the representation is dimension-dependent, we first consider the 2D formulation and then discuss the extension to the 3D case.
  
Using the auxiliary curve $\Lambda$, the scattered wave $\bm{v}$ can be represented in the form of \eqref{eq:split}, with the  scalar potential functions $\phi$ and $\psi$ given by the single-layer potential with densities $g_1, g_2,$ respectively, namely:
\begin{align}
& \label{eq:phi} \phi(x)=\int_{\Lambda}\Phi_p(x, y)g_1(y)\mathrm{d}s(y),\\
& \label{eq:psi} \psi(x)=\int_{\Lambda}\Phi_s(x, y)g_2(y)\mathrm{d}s(y).
\end{align}

Given the scattered field $\bm{v}=(v_1, v_2)^\top$ on $\Gamma_R,$ the scalar potential functions $\phi$ and $\psi$ are supposed to satisfy
$\nabla\phi+\mathbf{curl}\psi=\bm{v}$ on $\Gamma_R$, thus the density $\bm{g}=(g_1, g_2)^\top$ satisfies the following integral equation:
\begin{align}\label{eq:density}
(\mathcal{S}\bm{g})(x)=\bm{v}(x),\quad x\in\Gamma_R,
\end{align}
with the operator $\mathcal{S}:(L^2(\Lambda))^2\to(L^2(\Gamma_R))^2$ defined by
\begin{align*}
(\mathcal{S}\bm{g})(x)=\int_\Lambda
\mathbb{K}(x, y)\bm{g}(y)\mathrm{d}s(y),
\end{align*}
where $\bm{g}=(g_1,g_2)^\top\in(L^2(\Lambda))^2$, and the kernel is given by
\begin{align*}
\mathbb{K}(x, y)=
\begin{bmatrix}
\partial_{x_1}\Phi_p(x, y)&\partial_{x_2}\Phi_s(x, y)\\
\partial_{x_2}\Phi_p(x, y)&-\partial_{x_1}\Phi_s(x, y)
\end{bmatrix}.
\end{align*}
Accordingly, the adjoint operator $\mathcal{S}^*: (L^2(\Gamma_R))^2\to(L^2(\Lambda))^2$ is given by
$$
(\mathcal{S}^*\bm{\psi})(y)=\int_{\Gamma_R}\mathbb{K}^*(y, x)\bm{\psi}(x)
\mathrm{d}s(x),
$$
where $\mathbb{K}^*$ is the complex conjugate transpose of $\mathbb{K}$, and $\bm{\psi}=(\psi_1,\psi_2)^\top\in (L^2(\Gamma_R))^2.$

The integral operator $\mathcal{S}$ has an analytic kernel and therefore \eqref{eq:density} is severely ill-posed, which motivates us to apply the Tikhonov regularization to find the regularized density $\bm{g}^\xi=(g_1^\xi, g_2^\xi)^\top$ by solving
\begin{equation}\label{eq:regDensity}
(\xi \mathbb{I}_2+\mathcal{S}^*\mathcal{S})\bm{g}^\xi=\mathcal{S}^*\bm{v},
\end{equation}
where $\xi>0$ is the regularization parameter.

Once the regularized density function $\bm{g}^\xi=(g_1^\xi, g_2^\xi)^\top$ is obtained by solving \eqref{eq:regDensity}, the approximation $\bm{v}^\xi=(v_1^\xi, v_2^\xi)^\top$ for the scattered field $\bm{v}$ can be represented in form of
\[
\bm{v}^\xi=\nabla\phi^\xi+\bf{curl}\psi^\xi,
\]
with $\phi^\xi$ and $\psi^\xi$ obtained by inserting the regularized densities $g_1^\xi$ and $g_2^\xi$ into the single-layer potential representation \eqref{eq:phi} and \eqref{eq:psi}, respectively.
Explicitly, we obtain that
\begin{align}\label{eq:regscattered}
\bm{v}^\xi(x)=\int_\Lambda\mathbb{K}(x, y)\bm{g}^\xi(y)\mathrm{d}s(y).
\end{align}

Similar to the 2D case, the 3D formulation of the potentials can be derived as well. Assume that the scattered field $\bm{v}$ is split by \eqref{eq:split} into a scalar potential $\phi$ and a vector potential $\bm{\psi}$:
\begin{align}
& \label{eq:phi3D} \phi(x)=\int_{\Lambda}\Phi_p(x, y)g_p(y)\mathrm{d}s(y),\\
& \label{eq:psi3D} \bm{\psi}(x)=\frac{1}{k_s^2}\mathbf{curl}\mathbf{curl}\int_{\Lambda}\Phi_s(x, y)\bm{g}_s(y)\mathrm{d}s(y),
\end{align}
where $g_p\in L^2(\Lambda), \bm{g}_s=(g_{s1}, g_{s2}, g_{s3})^\top\in (L^2(\Lambda))^3$ are the scalar and vector densities,  respectively.

Denote by $\bm{\nu}=(\nu_1,\nu_2,\nu_3)^\top\in\mathbb{S}^2$ the unit normal vector to $\Gamma_R,$ and let $
\nabla\phi+\mathbf{curl}\bm{\psi}=\bm{v}$ on $\Gamma_R$. Taking the dot product and the cross product of the above equation with $\bm{\nu}$, respectively, we get
\begin{equation}\label{eq:density3D1}
\begin{cases}
\quad\bm{\nu}\cdot\nabla\phi+\bm{\nu}\cdot\mathbf{curl}\bm{\psi}=\bm{\nu}\cdot\bm{v}, & \text{on}\ \Gamma_R, \\
\bm{\nu}\times\nabla\phi+\bm{\nu}\times\mathbf{curl}\bm{\psi}=\bm{\nu}\times\bm{v}, & \text{on}\ \Gamma_R.
\end{cases}
\end{equation}
Using $\mathbf{curl}\mathbf{curl}=\nabla(\nabla\cdot)-\Delta,$ one easily derives that $\mathbf{curl}\bm{\psi}=\mathbf{curl}\int_{\Lambda}\Phi_s(x, y)\bm{g}_s(y)\mathrm{d}s(y)$. Consequently, substituting $\phi$ and $\bm{\psi}$ in \eqref{eq:density3D1} by \eqref{eq:phi3D}--\eqref{eq:psi3D}, a straightforward calculation shows that 
\begin{equation}\label{eq:density3D}
(\mathcal{S}\bm{g})(x)=\int_{\Lambda}\mathbb{K}(x, y)\bm{g}(y)\mathrm{d}s(y)=\bm{t}(x),\quad x\in\,\Gamma_R,
\end{equation}
where 
$$
\mathbb{K}=
\begin{bmatrix}
    \partial_{\bm{\nu}} \Phi_p    & (\bm{\nu}\times\nabla\Phi_s)^\top \\
    \bm{\nu}\times\nabla\Phi_p & \nabla\Phi_s\otimes\bm{\nu}-\partial_{\bm{\nu}}\Phi_s\mathbb{I}_3
\end{bmatrix}
,\quad
\bm{g}=
\begin{bmatrix}
    g_p\\
    \bm{g_s}
\end{bmatrix}
,\quad
\bm{t}=
\begin{bmatrix}
    \bm{\nu}\cdot\bm{v}\\
    \bm{\nu}\times\bm{v}
\end{bmatrix}.
$$

Since the operator $\mathcal{S}:(L^2(\Lambda))^4\to(L^2(\Gamma_R))^4$ in \eqref{eq:density3D} is well defined and it has an analytic kernel, the Tikhonov regularization strategy is employed to solve the ill-conditioned equation \eqref{eq:density3D}:
\begin{equation}\label{eq:regDensity3D}
(\xi \mathbb{I}_4+\mathcal{S}^*\mathcal{S})\bm{g}^\xi=\mathcal{S}^*\bm{t},
\end{equation}
where $\xi>0$ is the regularization parameter, $\mathbb{I}_4$ is the $4\times 4$ identity matrix, and $\mathcal{S}^*:(L^2(\Gamma_R))^4\to(L^2(\Lambda))^4$ is the adjoint operator of $\mathcal{S}$.

Given the regularized density $\bm{g}^\xi=\left(g_p^\xi, \bm{g}_s^\xi\right)^\top,$ the approximate scattered field $\bm{v}^\xi=(v_1^\xi, v_2^\xi, v_3^\xi)^\top$ can be given by 
\begin{align*}
\bm{v}^\xi(x) & =\nabla\phi^\xi(x)+\mathbf{curl}\bm{\psi}^\xi(x) \\
& = \nabla\int_\Lambda\Phi_p(x, y)g_p^\xi(y)\mathrm{d}s(y)+\frac{1}{k_s^2}\mathbf{curl}\mathbf{curl}\mathbf{curl}\int_\Lambda\Phi_s(x, y)\bm{g}_s^\xi(y)\mathrm{d}s(y)\\
& = \int_\Lambda \left( \nabla\Phi_p(x, y)g_p^\xi(y)+ \mathbf{curl}(\Phi_s(x, y)\bm{g}_s^\xi(y)) \right)\mathrm{d}s(y).
\end{align*}

\subsection{Linearization by explicit gradient}
Once the ansatz field $\bm{v}^\xi$ is available, the gradient can be evaluated explicitly. In 2D, the gradient of the approximate scattered field $\bm{v}^\xi$ can be  calculated by
\begin{align*}
\nabla\bm{v}^\xi &=
\begin{bmatrix}
\partial_{x_1}v_1^\xi & \partial_{x_2}v_1^\xi\\
\partial_{x_1}v_2^\xi & \partial_{x_2}v_2^\xi
\end{bmatrix},
\end{align*}
where
\begin{align*}
& \partial_{x_1}v_1^\xi=\int_\Lambda\left(\partial^2_{x_1x_1}\Phi_p(x, y)g_1^\xi(y)
+\partial^2_{x_1x_2}\Phi_s(x, y)g_2^\xi(y)\right)\mathrm{d}s(y), \\
& \partial_{x_2}v_1^\xi=\int_\Lambda\Big(\partial^2_{x_1x_2}\Phi_p(x, y)g_1^\xi(y)
+\partial^2_{x_2x_2}\Phi_s(x, y)g_2^\xi(y)\Big)\mathrm{d}s(y),  \\
& \partial_{x_1}v_2^\xi=\int_\Lambda\Big(\partial^2_{x_1x_2}\Phi_p(x, y)g_1^\xi(y)
-\partial^2_{x_1x_1}\Phi_s(x, y)g_2^\xi(y)\Big)\mathrm{d}s(y), \\
& \partial_{x_2}v_2^\xi=\int_\Lambda\left(\partial^2_{x_2x_2}\Phi_p(x, y)g_1^\xi(y)
-\partial^2_{x_1x_2}\Phi_s(x, y)g_2^\xi(y)\right)\mathrm{d}s(y).
\end{align*}

Analogously, the gradient of the approximate scattered field $\bm{v}^\xi$ in $\mathbb{R}^3$ can be explicitly calculated as the $3\times 3$ tensor $\nabla^\top\bm{v}^\xi$ whose $(i, j)$-th entry is given by $\left(\partial_{x_j}v_i^\xi\right),\ i, j=1,2,3$.

Based on the gradient of the approximate scattered field, we are now ready to propose the Newton-type iteration scheme for recovering $\partial D$. Given the incident field $\bm{u}^i$ and the approximate scattered field $\bm{v}^\xi,$ we define the operator $F^\xi$ mapping the boundary contour $\gamma$ onto the approximate total field $\bm{u}^\xi=\bm{u}^i+\bm{v}^\xi=(u_1^\xi, \cdots, u_d^\xi)^\top$, for later use, we denote
\begin{equation}\label{eq:def_F}
F^\xi:\gamma\mapsto \bm{u}^\xi.
\end{equation}
To seek the Dirichlet boundary where the total field $\bm{u}^\xi$ vanishes, we only need to find the parameterization $p$ of the boundary contour $\gamma$ such that
\begin{equation}\label{Gz=0}
F^\xi(p)=\bm{0},\quad p\in\gamma.
\end{equation}

Now, we consider the linearization of the above equation. Let $\gamma_j,\,j=0,1,\cdots, n$ be the approximation to the boundary $\partial D$, we want to seek for $\gamma_{n+1}$ such that $F(p_{n+1})=0.$ Noticing that the mapping $F$ is nonlinear, we instead update the approximation via the following procedure
\begin{equation}\label{eq:Newton}
\left\{
\begin{aligned}
&F^\xi(p_n)+F'^\xi(p_n)h_n=\bm{0},\\
&p_{n+1}=p_n+h_n,
\end{aligned}
\right.
\end{equation}
where $h_n$ is the shift at the $n$-th iteration and the gradient $F'^\xi=\nabla^\top\bm{u}^\xi$ in \eqref{eq:Newton} is a $d\times d$ tensor whose elements can be explicitly computed by
$$
\left[F'^\xi\right]_{i j}=\partial_{x_j}u_i^\xi,\quad i, j=1,\cdots,d.
$$
\begin{remark}
	The above procedure can be readily extended to the case of more than one incident wave. For example, here we briefly mention in passing the 3D modifications concerning multiple sources. The other similar details are omitted. 
	
	Let $\bm{v}_j(j=1,\cdots, N)$ denote the scattered field due to incident wave $\bm{u}^i_j(j=1,\cdots, N)$. Correspondingly, the vector functions $\bm{g}$ and $\bm{t}$ in \eqref{eq:density3D} should be replaced by the matrix-valued functions
	$$
	\bm{g}=
	\begin{bmatrix}
	\bm{g}_1 \cdots \bm{g}_N
	\end{bmatrix},
	\quad
	\bm{t}=
	\begin{bmatrix}
	\bm{\nu}\cdot\bm{v}_1 & \cdots & \bm{\nu}\cdot\bm{v}_N \\
	\bm{\nu}\times\bm{v}_1 & \cdots &  \bm{\nu}\times\bm{v}_N
	\end{bmatrix}.
	$$
	Then the approximate scattered field $\bm{v}_j^\xi$ can be represented by the layer potential with regularized density $\bm{g}_j^\xi(j=1,\cdots, N)$. Accordingly, the operator $F^\xi$ is column-wisely extended to the form
	$$
	F^\xi:\gamma\mapsto \bm{u}^\xi=[\bm{u}^\xi_1 \cdots \bm{u}^\xi_N],
	$$
	where $\bm{u}_j^\xi=\bm{u}_j^i+\bm{v}_j^\xi(j=1,\cdots, N)$ are the approximate total fields.
\end{remark}


\subsection{Star-like approximation}
To accomplish the iteration numerically, we still need an appropriate representation of the admissible surface $p(t)$ for approximating $\partial D$. Therefore, we briefly describe the 2D star-like approximation of the boundary curve, which is in the parametric form 
\[
p(t)=\{r(t)(\cos t,\sin t): t\in[0,2\pi]\},
\]
with $r\in C^2([0,2\pi], \mathbb{R}_+)$ denoting the radial function. For simplicity, we also denote by $r$ the approximation to $\partial D$ in what follows. To numerically approximate $r$, we assume that $r$ is represented as the trigonometric polynomials of degree less than or equal to $M\in\mathbb{N}_+,$ i.e.,
$$
r(t)=a_0+\sum_{\ell=1}^M(a_\ell\cos \ell t+b_\ell\sin\ell t).
$$
 
For an iteration sequence $\{r_n\}$ of radial functions, the shift at the $n$-th step is correspondingly written as $r_n^h$, i.e., $r_{n+1}=r_n+r_n^h, n=1, 2, \cdots$. Alternatively, if we denote by $\bm{c}= (a_0, a_1, \cdots,  a_M, b_1, \cdots,  b_M)^\top$ the Fourier coefficients to $r$, then the iteration process is implemented via the update $\bm{c}_{n+1}=\bm{c}_n+\bm{c}_n^h$ where $\{\bm{c}_n\}$ and $\{\bm{c}_n^h\}$ are the Fourier coefficients to $r_n$ and $r_n^h$, respectively.

Now, the iteration procedure \eqref{eq:Newton} can be rewritten as
$$
\begin{cases}
\bm{u}^\xi(r_n\hat{x})+(\nabla^\top\bm{u}^\xi(r_n\hat{x})\hat{x})r_n^h=\bm{0},\\
r_{n+1}=r_n+r_n^h,
\end{cases}
$$
or more specifically,
$$
\begin{cases}
\bm{u}^\xi(B\bm{c}_n\hat{x})+(\nabla^\top\bm{u}^\xi(B\bm{c}_n\hat{x})\hat{x})B\bm{c}_n^h=\bm{0},\\
\bm{c}_{n+1}=\bm{c}_n+\bm{c}_n^h,
\end{cases}
$$
where $\hat{x}(t)=(\cos t, \sin t)^\top$ and $B(t)=(1, \cos t, \cdots, \cos Mt, \sin t, \cdots, \sin Mt)$.  For the further discretization, suppose that $t_j=2\pi j/J, j=1,\cdots, J,$ is a set of quadrature points, then the update $\bm{c}_n^h$ can be obtained by solving the linear system
$$
(\nabla^\top\bm{u}^\xi(B(t_j)\bm{c}_n\hat{x}(t_j))\hat{x}(t_j))B(t_j)\bm{c}_n^h=-\bm{u}^\xi(B(t_j)\bm{c}_n\hat{x}(t_j)),\quad j=1,\cdots, J.
$$

Typically for iterative algorithms, we finally need to impose a stopping criterion to terminate the iteration process. For convenience, we quantify the convergence of iteration by the relative error
\begin{equation}\label{eq:error}
E_n=\frac{\|h_n\|_{L^2}}{\|p_{n-1}\|_{L^2}},
\end{equation}
and choose some constant $\varepsilon>0.$ Once $E_n<\varepsilon$, the update process can be stopped.
For a description of the inversion algorithm to determine the obstacle, we refer to Algorithm 2.

\begin{table}
	\centering
	\begin{tabular}{cp{.8\textwidth}}
		\toprule
		\multicolumn{2}{l}{{\bf Algorithm 2:}\quad Newton-type method for the inverse elastic obstacle problem.} \\\hline
		{\bf Step 1} & Generate the approximate scattered data on $\Gamma_R$ by subtracting the exact or approximate incident field from the total field measurements $\mathbb{U}$; \\
		{\bf Step 2} & Select an auxiliary surface $\Lambda\subset D$ and represent the approximate scattered field as layer potentials by solving \eqref{eq:regDensity} or \eqref{eq:regDensity3D}; \\
		{\bf Step 3} & Choose an initial guess $\gamma^{(0)}$ for $\partial D$ and the error tolerance $\varepsilon.$ Set $n=1$; \\
		{\bf Step 4} &Solve \eqref{eq:Newton} to update the approximation $p_n$ and evaluate the error $E_n$;\\
		{\bf Step 5} &If $E_n>\varepsilon,$ then set $n=n+1$ and go to {\bf Step 4}. Otherwise, take the current approximation $p_n$ as the final reconstruction of $\partial D.$
		\\\bottomrule
	\end{tabular}
\end{table}


\section{Numerical experiments}\label{sec:numerical}
In this section, we shall conduct several numerical experiments to verify the efficiency and effectiveness of the proposed methods. 
In our numerical experiments, the synthetic scattered fields are generated by reformulating the direct problem into coupled boundary integral equations (cf. \cite{Dong1, Dong2}), and the integral equations are solved by the Nystr\"{o}m method based on Alpert's quadrature. The receivers are chosen to be $x_r=10(\cos\theta_r,\sin\theta_r),\,\theta_r=\pi r/60,\,r=1,2,\cdots, 120.$ The forward solver produces the total field data
$$
\bm{u}(x_r; z_j),\quad r=1,\cdots, 120,\ j=1,\cdots, N.
$$

To test the stability of the proposed algorithm, random noise is added to the measured data. The noisy total field data are given according to the following formula:
\begin{align*}
\bm{u}^\epsilon=\bm{u}+\epsilon r_1|\bm{u}|\mathrm{e}^{\mathrm{i}\pi r_2},
\end{align*}
where $r_1,\,r_2$ are two uniformly distributed random numbers ranging from $-1$ to $1,$ and $\epsilon>0$ is the noise level.

In the following experiments, the Lam\'{e} constants are set to be $\lambda=\mu=1$. The sampling domain for locating the source points is set to be $\Omega = [-5,5]\times[-5,5]$ with $200\times200$ equally spaced sampling grid. 
For the purpose to determine the polarization $\bm{p},$ we set $N_q=40$ in {\bf Step 3} of Algorithm 1, i.e., we seek for $\bm{p}$ over a sampling angles $\theta_\ell=\frac{\ell\pi}{40},\,\ell=0,1,\cdots, 39$.
 The noise level is set to be $\epsilon=5\%$ unless otherwise stated. The auxiliary curve is chosen to be a circle centered at the origin with a radius of $0.7.$ The integrals over the auxiliary curve are numerically approximated by the trapezoidal rule with 100 grid points. The regularization parameter $\xi$ is set to be $10^{-2}$. The boundaries of the obstacles are parameterized as follows:
\begin{align}
L\text{-leaf:}\quad&x(t)=(1+0.2\cos Lt)(\cos t,\sin t),\quad 0\le t\le 2\pi,\quad L=3, 5,\\
\text{Kite:}\quad&x(t)=(\cos t+0.65\cos2t-0.65,1.5\sin t),\quad 0\le t\le 2\pi.
\end{align}

\subsection{Inverse source problem}\label{subsec:inverse_source}
In this subsection, we test the performance of Algorithm 1 in reconstructing the source from the measured total field, i.e., we now consider the source identification with the unknown obstacle.

\begin{example}
	In the first example, we consider the reconstruction of the locations and the polarizations of the source points from the noisy total field data. The angular frequency is set to be $\omega=8.$ The two initial polarizations are chosen to be $\bm{q}_1=\frac{1}{\sqrt{2}}(1,1)^\top$ and $\bm{q}_2=\frac{1}{\sqrt{2}}(-1,1)^\top$. The source locations and polarizations together with the reconstructions are displayed in \Cref{tab:source}. As shown in \Cref{tab:source}, all the locations and the polarizations are well-reconstructed from the noisy total field data. It deserves noting that, though the polarizations differ from point to point, all the polarizations are well-recovered.
	
	Further, we point out that it is necessary to determine the source positions and the polarizations twice using different auxiliary polarizations $\bm{q}_i,\,i=1,2,$ as done in \Cref{subsec:polarization}. Otherwise, it may fail to reconstruct the source (locations or polarizations) correctly. To clarify this point, we further consider the reconstruction of the four source points and the associated polarizations in \Cref{tab:source2}--\Cref{tab:source3}.
	\Cref{tab:source2} shows that the locations can be well reconstructed. One can find from \Cref{tab:source3} that the reconstruction of the polarization significantly depends on the choice of $\bm{q}_i$. However, after a second calibration as introduced in \Cref{subsec:polarization}, the polarizations can be also stably and well recovered. These results illustrate that our method performs well in recognizing the source position and the associated polarization from the noisy total field data. 
\end{example}

\begin{table}[!htb]
	\begin{center}
		\caption{Comparison of the locations and polarizations.}\label{tab:source}
		\label{tab:table1}
		\begin{tabular}{l|ll|ll}\toprule
			&\multicolumn{2}{c|}{Exact sources} & \multicolumn{2}{c}{Reconstructed sources}\\ 
			\cline{2-5} 
			&Locations & polarizations &Locations & polarizations \\\hline
			Point 1 & $(3,0)$ & $(0.86,0.5)$ & $(2.98,-0.03)$&$(0.89,0.45)$\\
			Point 2 & $(1.5,2.59)$ & $(0.5,0.86)$ & $(1.47,2.56)$&$(0.52,0.85)$\\
			Point 3 & $(-1.5,2.59)$ & $(0,1)$ & $(-1.47,2.62)$&$(-0.07,0.99)$\\
			Point 4 & $(-3,0)$&$(-0.5,0.86)$&$(-3.04,0.03)$&$(-0.52,0.85)$\\
			Point 5 & $(-1.5,-2.59)$&$(-0.86,0.5)$&$(-1.53,-2.62)$&$(-0.85,0.52)$\\
			\bottomrule
		\end{tabular}
	\end{center}
\end{table}
\begin{table}[!htb]
\begin{center}
	\caption{Reconstruction of the four point sources (locations).}\label{tab:source2}
	\begin{tabular}{c|cccc}\toprule
		& Point 1 & Point 2 & Point 3 & Point 4\\\hline
	Exact & $(3,0)$ & $(0,3)$ & $(-3,0)$ & $(0,-3)$\\
	Reconstructed & $(2.98, -0.03)$ & $(-0.03, 2.98)$ & $(-3.04, 0.03)$ & $(-0.03, -2.98)$\\\bottomrule
	\end{tabular}
\end{center}

\end{table}
\begin{table}[!htb]
\begin{center}
	\caption{Reconstruction of the four point sources (polarizations).}\label{tab:source3}
	\begin{tabular}{ll|lll}\toprule
		& polarizations & $\bm{q}_{\ell_j^1}$ & $\bm{q}_{\ell_j^2}$ & $\tilde{\bm{q}}_j$\\ \hline
		Point 1 & $(0.76,0.64)$ & $(0.76,0.64)$ & $(-0.99,0.07)$ & $(0.76,0.64)$\\
		Point 2 & $(0,1)$ & $(0,1)$ & $(0,1)$ & $(0,1)$\\
		Point 3 & $(-0.71,0.71)$ & $(0.99,0.07)$ & $(-0.71,0.71)$ & $(-0.71,0.71)$\\
		Point 4 & $(-1,0)$ & $(0.80,0.58)$ & $(-1,0)$ & $(-1,0)$\\
		\bottomrule
	\end{tabular}
\end{center}
\end{table}

\subsection{Inverse obstacle scattering problem}\label{subsec:inverse_obstacle}
In this subsection, we test the performance of Algorithm 2 in recovering the obstacle from the scattered field under the assumption that the sources are known in advance. In all the figures in this subsection, the red solid lines represent the exact boundaries, and the black dashed lines denote the reconstructed boundaries.
\begin{example}
	In this example, we adopt the Newton-type method proposed in \Cref{sec:Newton} to recover the obstacles of 3-leaf shape and kite shape.	 Here 12 source points are equally placed on the measurement circle with radius $R=3.$ The polarization is chosen to be $\bm{p}=(\cos(\pi/3),\sin(\pi/3)).$
\end{example}

We first consider the reconstruction of the 3-leaf shaped obstacle. By taking $\omega=5,\,8$ and 10, we display the reconstruction for the 3-leaf shaped obstacle in \Cref{fig:pear_obstacle}. From \Cref{fig:pear_obstacle}, we find that the 3-leaf obstacle can be well-reconstructed with three angular frequencies.

Further, we consider the reconstruction of the kite-shaped obstacle. By taking $\omega$ to be $3,5$ and $6,$ we display the reconstruction of the kite-shaped obstacle in \Cref{fig:kite_obstacle}--\Cref{fig:kite_obstacle2}. It can be seen that the reconstructions are accurate.

In \Cref{fig:pear_obstacle}--\Cref{fig:kite_obstacle2}, the green dashed lines denote the initial guesses for the Newton-type methods. We can easily see from the results that the method performs well no matter whether the obstacle is starlike or not. Especially for the kite-shaped obstacle, the two wings and the concave regions are well recovered using the proposed method.

In addition, the relative error estimator $E_n$ defined in \eqref{eq:error} is plotted against the number of iterations in the second row of \Cref{fig:kite_obstacle}. It can be seen from the error curves that the relative error estimator $E_n$ decreases quickly, which demonstrates that our method converges fast.

\begin{figure}
	\centering
	\subfigure[]{\includegraphics[width=0.3\textwidth]{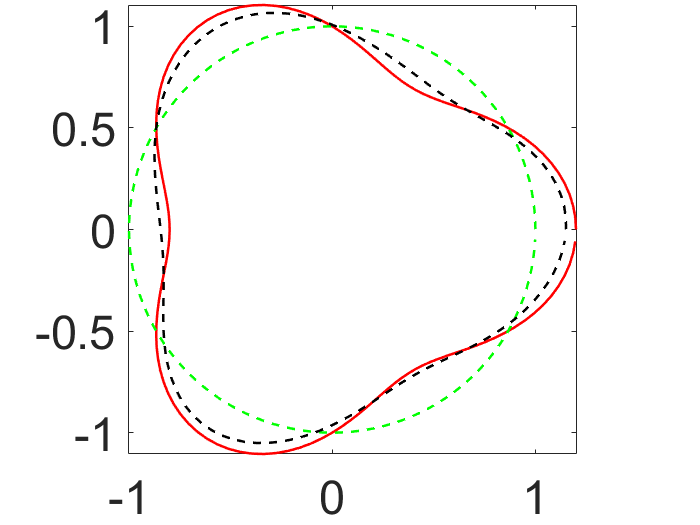}}
	\subfigure[]{\includegraphics[width=0.3\textwidth]{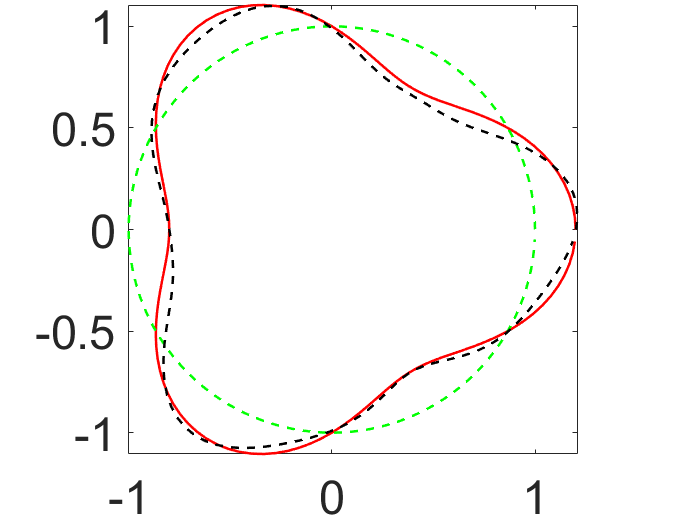}}
	\subfigure[]{\includegraphics[width=0.3\textwidth]{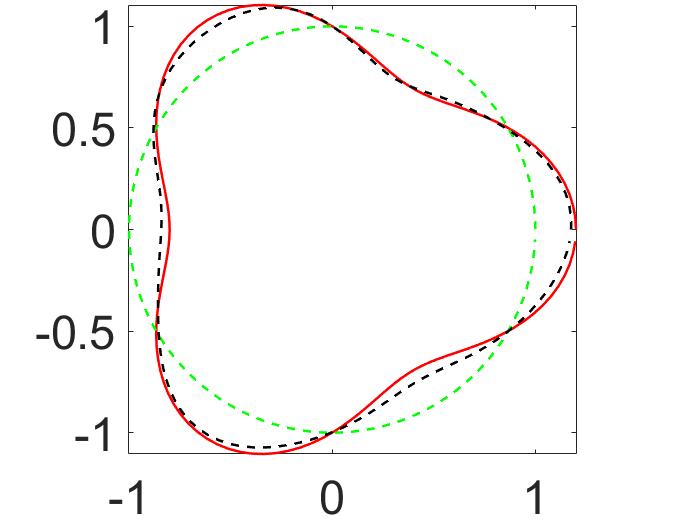}} \\	
	\subfigure[]{\includegraphics[width=0.3\textwidth]{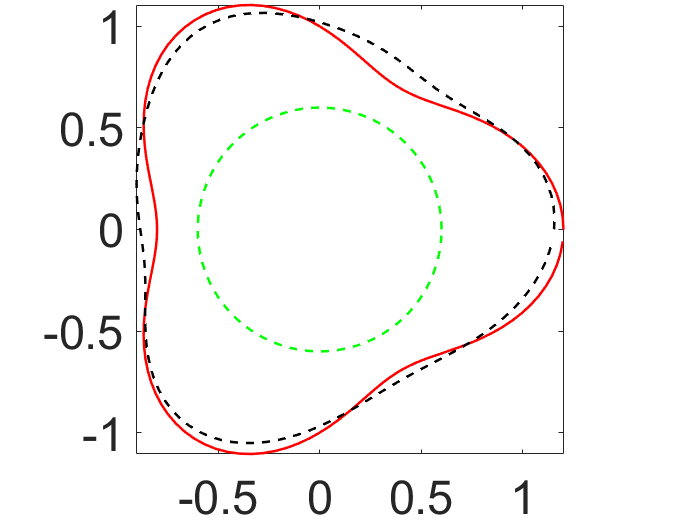}}
	\subfigure[]{\includegraphics[width=0.3\textwidth]{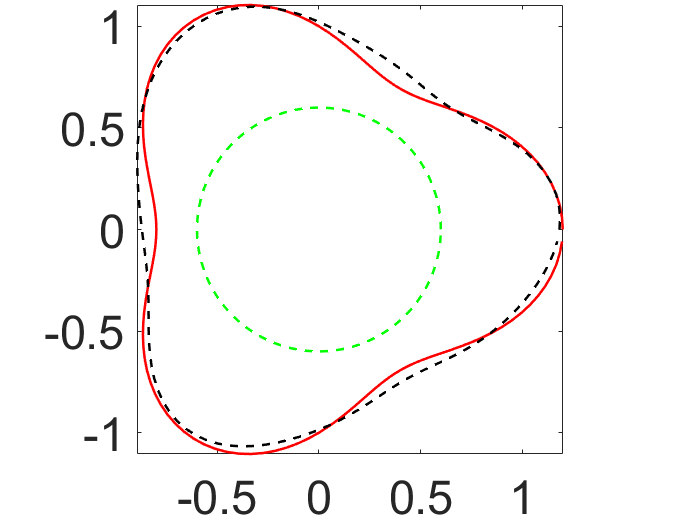}}
	\subfigure[]{\includegraphics[width=0.3\textwidth]{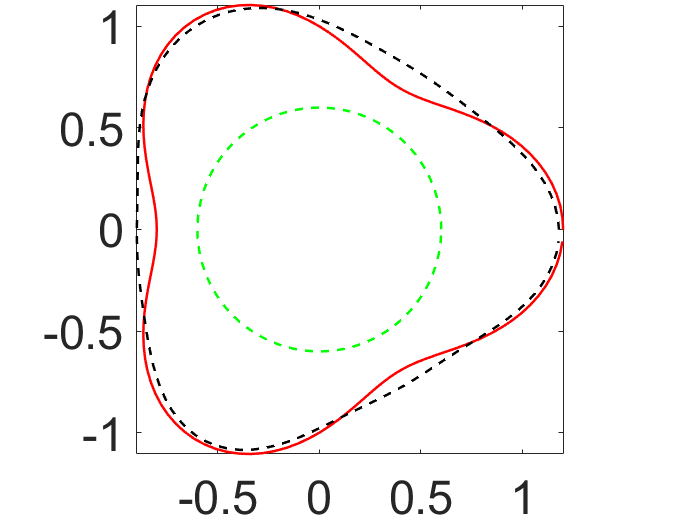}}
	\caption{Reconstruction of the 3-leaf obstacle with different frequencies. Left column: $\omega=6$; Middle column: $\omega=8$; Right column: $\omega=10$.}\label{fig:pear_obstacle}
\end{figure}

\begin{figure}
	\centering
	\subfigure[]{\includegraphics[width=0.3\textwidth]{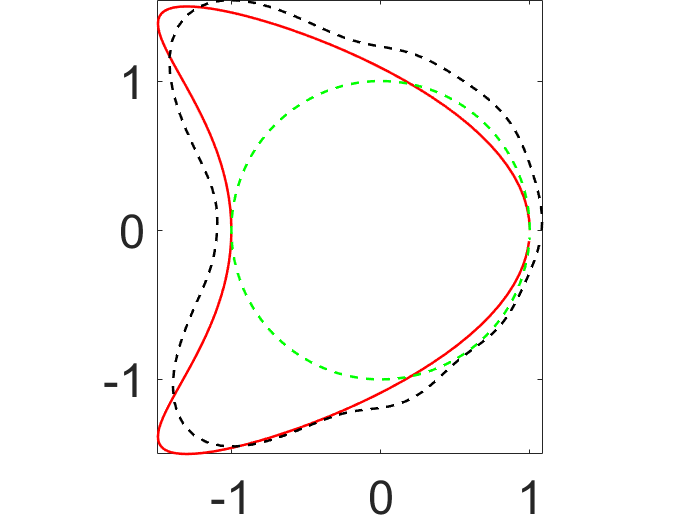}}
	\subfigure[]{\includegraphics[width=0.3\textwidth]{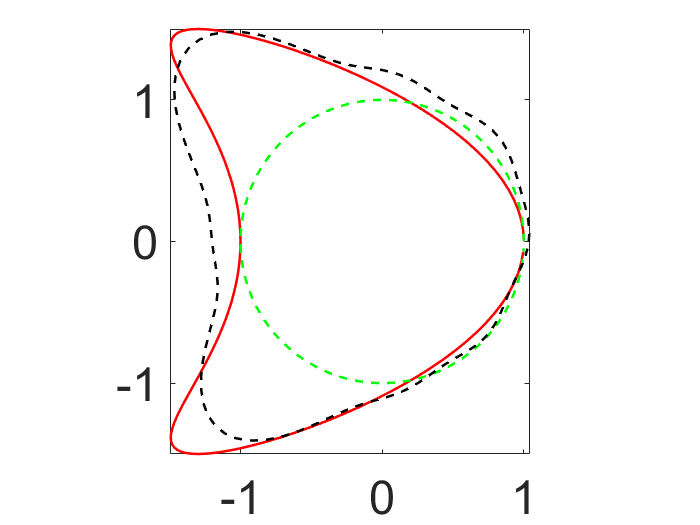}}
	\subfigure[]{\includegraphics[width=0.3\textwidth]{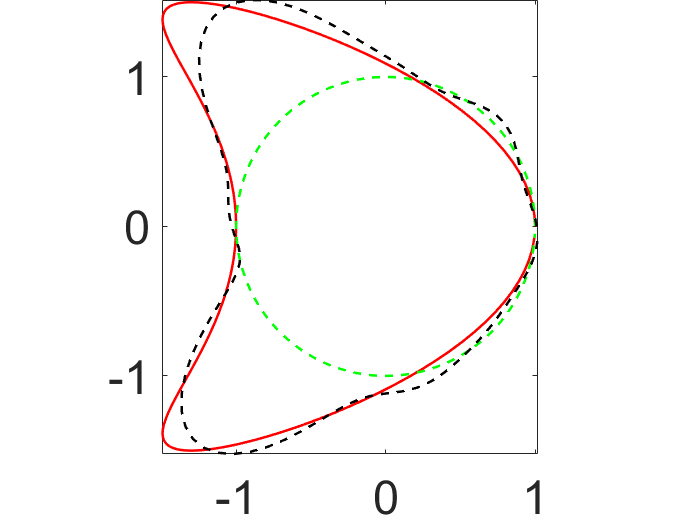}}
	
	\subfigure[]{\includegraphics[width=0.3\textwidth]{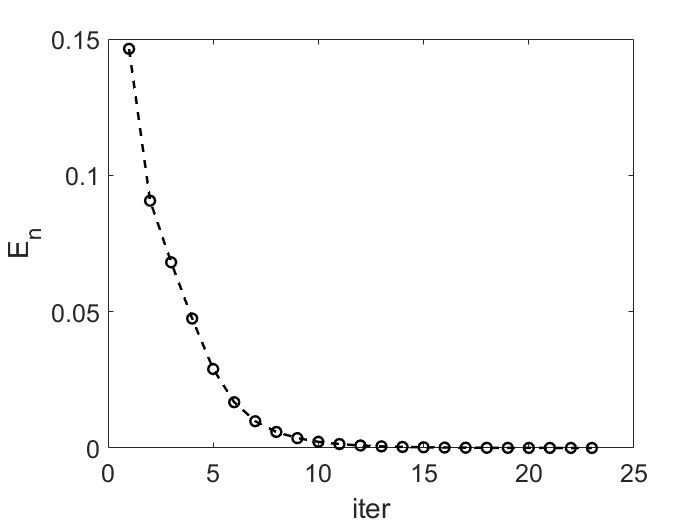}}
	\subfigure[]{\includegraphics[width=0.3\textwidth]{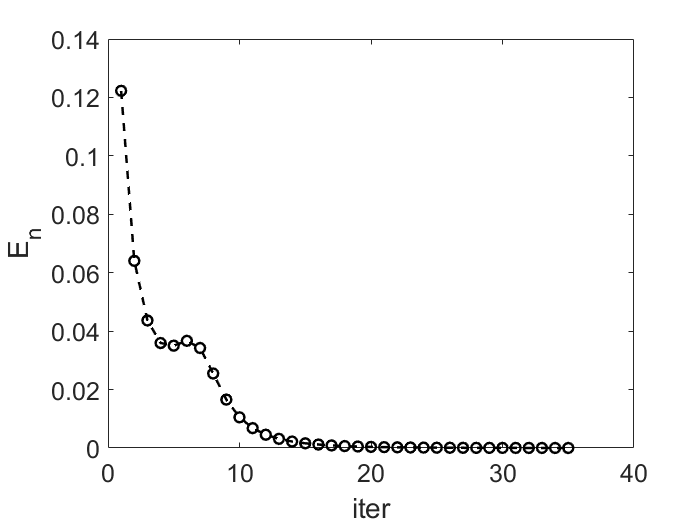}}
	\subfigure[]{\includegraphics[width=0.3\textwidth]{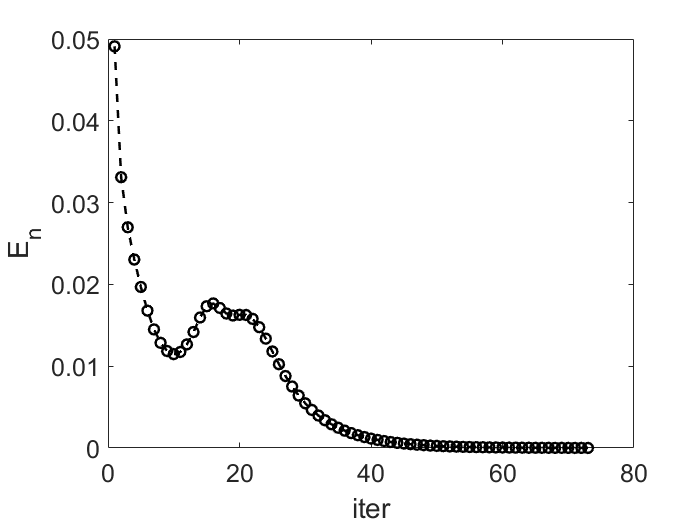}}
	\caption{Reconstruction of the kite-shaped obstacle with different frequencies. Left column: $\omega=3$; Middle column: $\omega=4$; Right column: $\omega=6$.}\label{fig:kite_obstacle}
\end{figure}

\begin{figure}
	\centering
	\subfigure[]{\includegraphics[width=0.3\textwidth]{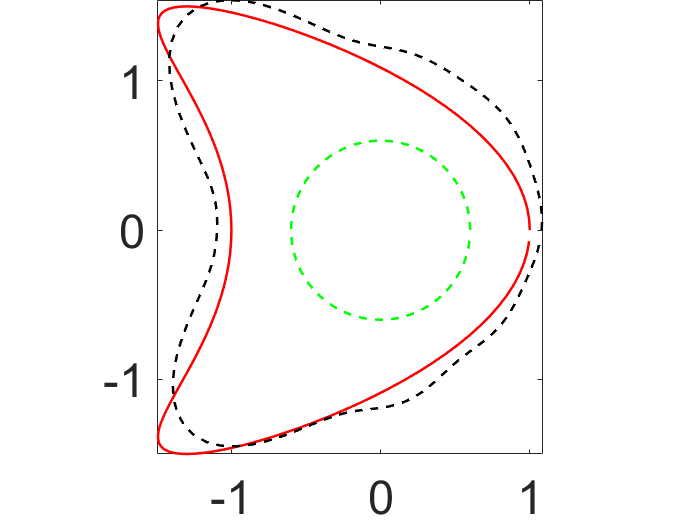}}
	\subfigure[]{\includegraphics[width=0.31\textwidth]{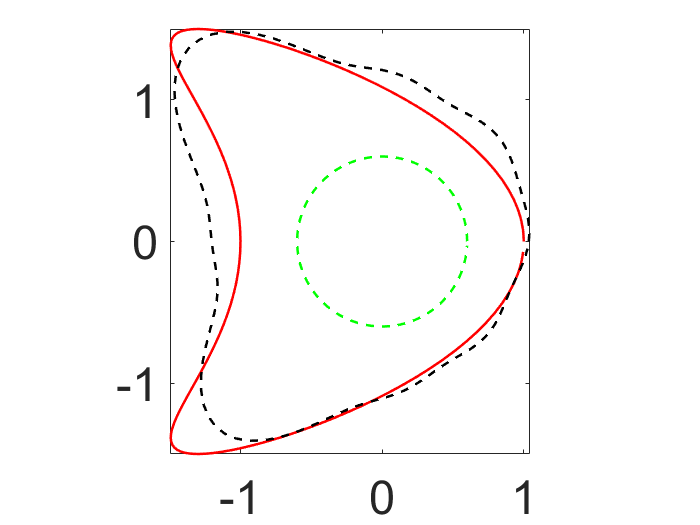}}
	\subfigure[]{\includegraphics[width=0.3\textwidth]{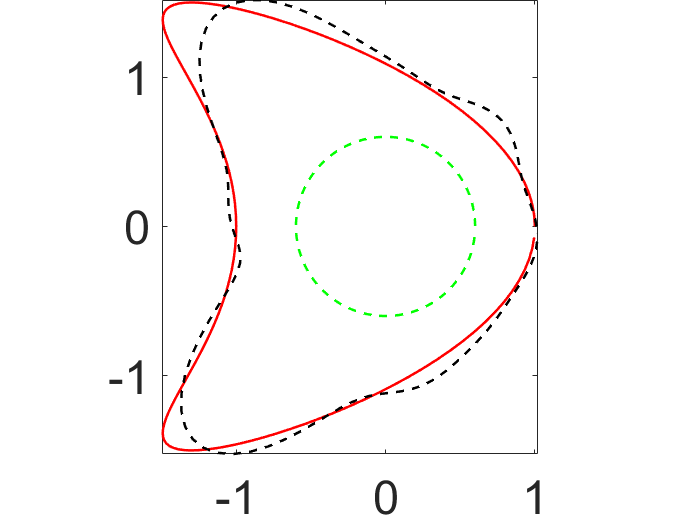}}
	\caption{Reconstruction of the kite-shaped obstacle with initial guess different from that in \Cref{fig:kite_obstacle}. (a) $\omega=3$; (b) $\omega=4$; (c) $\omega=6$.}\label{fig:kite_obstacle2}
\end{figure}

\begin{example}
	In the second example, we consider the reconstruction of star-like obstacles whose radial functions $r(t)$ of the boundary curves are randomly generated by the following procedure:
\begin{itemize}
		\item Choose $N_t$ randomly in the integer set $\{8,9,\cdots, 20\}$ and generate the knots $T_t=2t\pi/N_t,\,t=0,1,\cdots, N_t-1$;
		\item For each $T_t,$ we generate randomly the radial grid knots $r(T_t)\sim\mathcal{U}[0.8,1.8],$ i.e., $r(T_t)$ is a uniform distribution in $[0.8,1.8].$ Under this setting, the boundary curves are located in the annular domain with the inner radius $0.8$ and outer radius  $1.8$;
		\item Given the random radial grid knots $(T_t, r(T_t)),\,t=0,1,\cdots, n_T-1,$ we generate the radial function $r(t)$ by the cubic spline interpolation  to $(T_t, r(T_t)).$ In addition, to ensure that the generated star-like curves are closed, the periodic condition $r(T_0)=r(T_{n_T})$ is imposed;
\end{itemize}

We refer to \Cref{fig:Random} for some examples of such randomly generated shapes. In \Cref{fig:Random}, the red solid lines stand for the boundaries of the obstacle, the black small circles designate the radial gird knots $r(T_t)$ generated randomly, the black solid lines mark the radial and the blue dashed circles with radii 0.8 and 1.4 respectively bound the domains containing the boundary curves.
\end{example}

\begin{figure}
	\centering
	\subfigure[]{\includegraphics[width=0.3\textwidth]{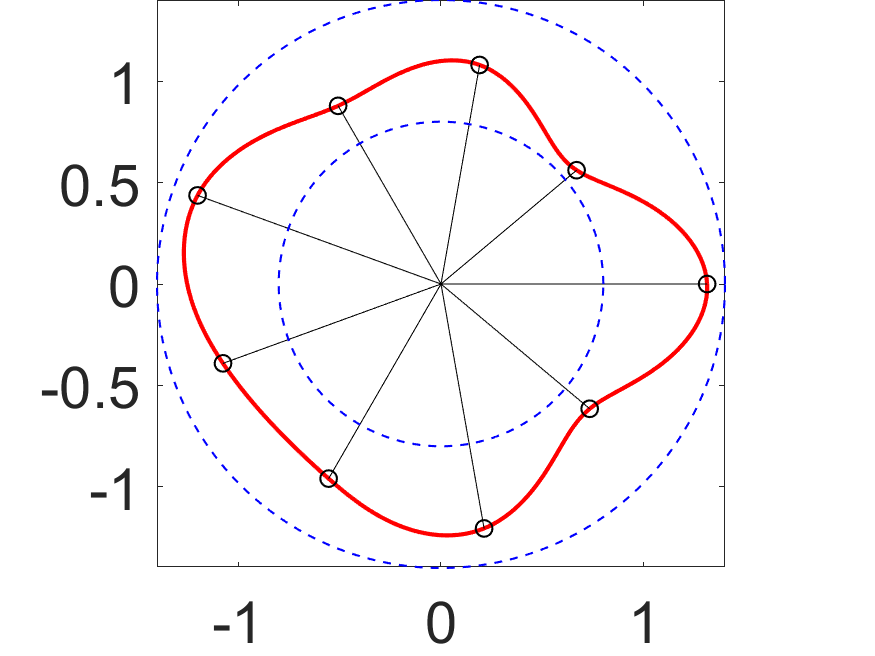}}
	\subfigure[]{\includegraphics[width=0.3\textwidth]{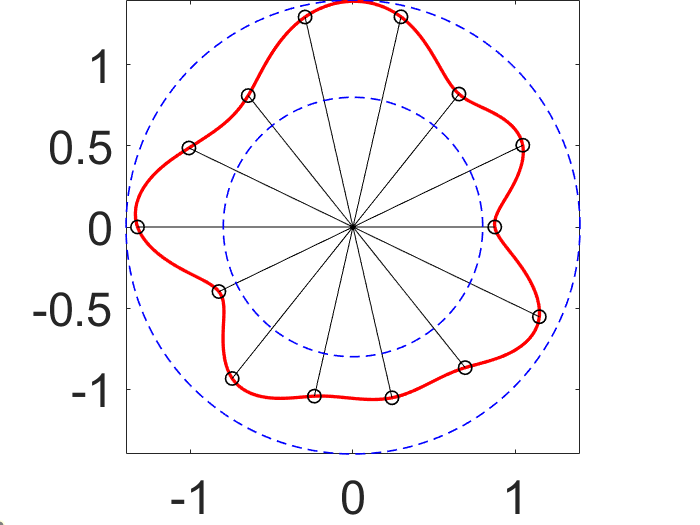}}
	\subfigure[]{\includegraphics[width=0.3\textwidth]{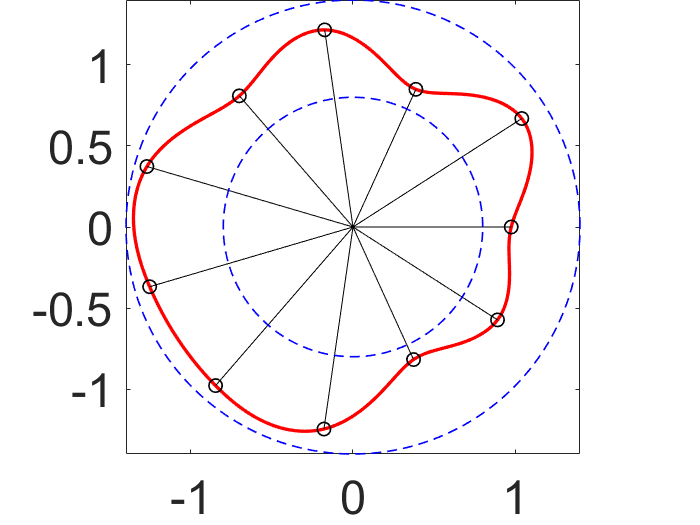}}
	\caption{Illustration of the randomly generated shapes.}\label{fig:Random}
\end{figure}

By taking $\omega=9,$ we exhibit the reconstruction subject to different initial guesses in \Cref{fig:RandomResult}.
From \Cref{fig:RandomResult}, we find that the non-symmetric obstacle can be well reconstructed by the novel Newton-type method.  Though with different initial guesses, the algorithms perform satisfactorily. In addition, we find that the convex parts of the non-symmetric obstacles are better reconstructed compared with the non-convex part, which is because the convex domain can be illuminated more adequately and the total field brings us more geometry information.
\begin{figure}
	\centering
	\subfigure[]{\includegraphics[width=0.3\textwidth]{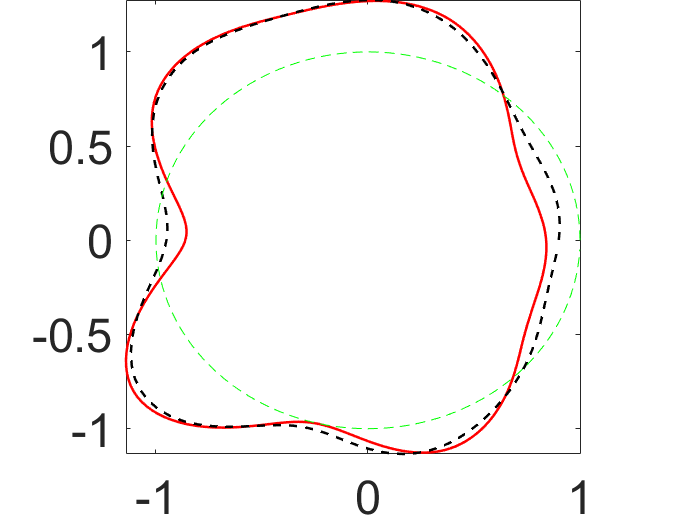}}
	\subfigure[]{\includegraphics[width=0.31\textwidth]{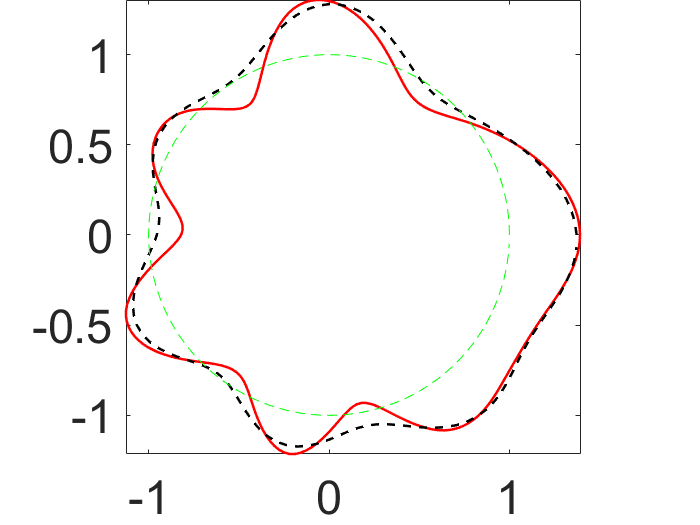}}
	\subfigure[]{\includegraphics[width=0.3\textwidth]{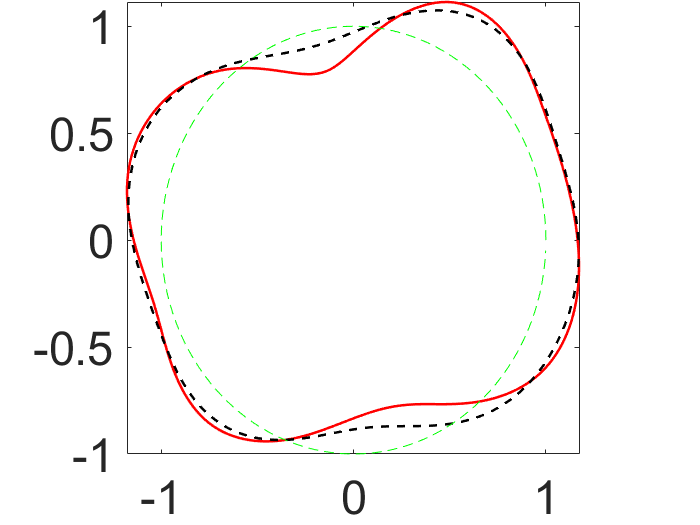}}
	\subfigure[]{\includegraphics[width=0.3\textwidth]{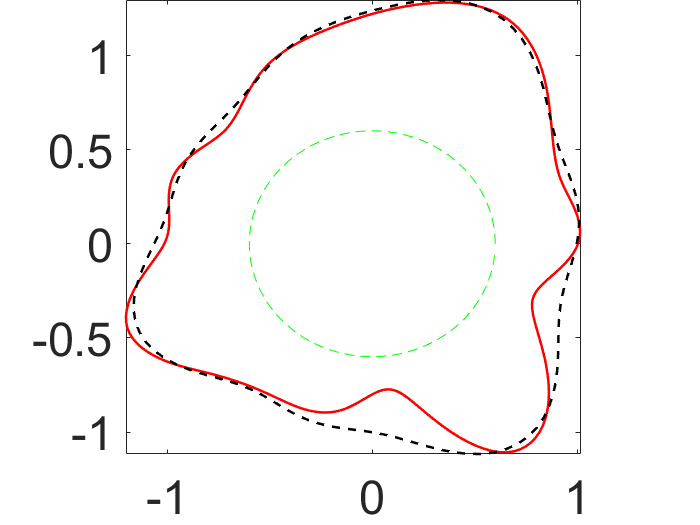}}
	\subfigure[]{\includegraphics[width=0.31\textwidth]{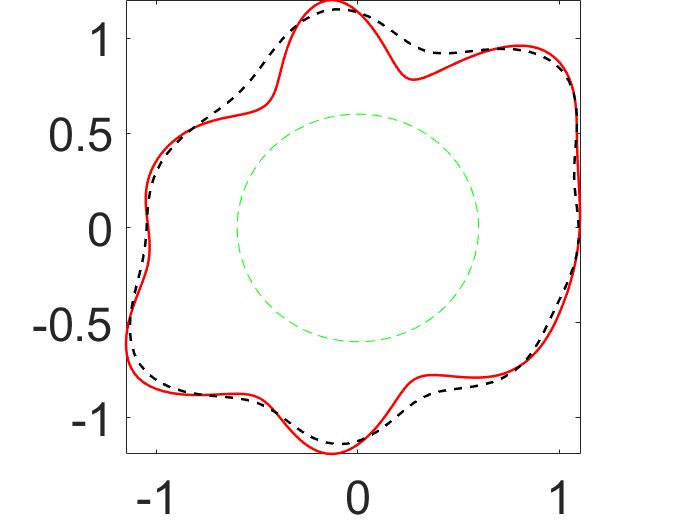}}
	\subfigure[]{\includegraphics[width=0.3\textwidth]{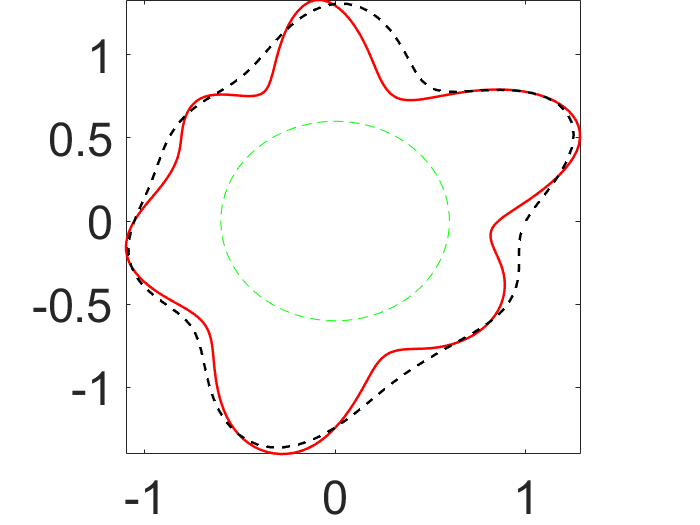}}
	\caption{Reconstruct the obstacles of non-symmetric shapes.}\label{fig:RandomResult}
\end{figure}

In this subsection, we test the performance of the novel Newton method by reconstructing the obstacles from the scattered field data through two numerical experiments. Though no forward solver is involved in the Newton method, the reconstruction is rather accurate.


\subsection{Co-inversion problem}\label{subsec:co_inversion}
By integrating Algorithm 1 with Algorithm 2, this subsection is devoted to the co-inversion of the obstacle-source pair from the total field. In this subsection, we assume that the polarization direction is fixed for all the sources in each example and concentrate on the reconstruction of the source locations and the shape of the obstacle. Throughout the subsection, the initial guess for the obstacle is chosen to be the unit circle centered at the origin. The polarizations are chosen to be $\bm{p}=(\cos(\pi/5),\sin(\pi/5))^\top$ and $\bm{q}=(\cos(\pi/4),\sin(\pi/4))^\top.$

In the figures about the reconstruction, the green solid lines represent the exact boundary of the obstacle, the blue dashed lines stand for the reconstructed boundary, the black `$+$' markers, and the small red circles denote the exact and the reconstructed source points, respectively.

\begin{example}
In this example, we consider the reconstruction of the $L$-leaf ($L=3,\,5$) shaped obstacle together with its excitation source points. 

In \Cref{fig:pear-iosp3}, we exhibit the reconstruction of the 3-leaf obstacle and three source points. In \Cref{fig:pear-iosp3}(a)--\Cref{fig:pear-iosp3}(c), we display the imaging function $I_j(y),\,j=1,2,3$ over the sampling domain $\Omega.$ It can be easily seen that each imaging function attains its peak at the source points, and the obstacle is reconstructed overall. As will be seen in the later numerical experiments, when more source points are determined, the quality of reconstruction for the obstacle can be improved.

By taking different angular frequencies, we exhibit the reconstruction of different obstacle-source pairs in \Cref{fig:pear-iosp}, which illustrates that both the obstacle and the source points can be well reconstructed. In addition, it can be seen that the distribution of source points influences the accuracy of obstacle reconstruction. This phenomenon naturally reflects the wave fields interplay between the source and the scatterer.

Next, we compare the reconstructed polarization directions with the exact one $\bm{p}=(\cos\frac{\pi}{5},\sin\frac{\pi}{5})\approx(0.8090, 0.5878)$. In \Cref{tab:polar}, we list the polarizations corresponding to \Cref{fig:pear-iosp}.  \Cref{tab:polar} shows that the polarization can be also well-reconstructed. Therefore, our method has the capability of reconstructing the obstacle, its excitation source points as well as the polarization direction from the total field.

In \Cref{fig:pear-iosp2}, we investigate the situation where the sources are spatially confined to a limited-angle sector region with respect to the obstacle. Clearly, all the source points are well recovered but only the illuminated part of the obstacle can be easily reconstructed. The recoveries of the `shadow regions' are of reasonably low resolution because the scatterer is not encircled by the sources and thus the geometry information is inadequately perceived by the sensors from the back. 
\end{example}

\begin{figure}
	\centering
	\subfigure[]{\includegraphics[width=0.24\textwidth]{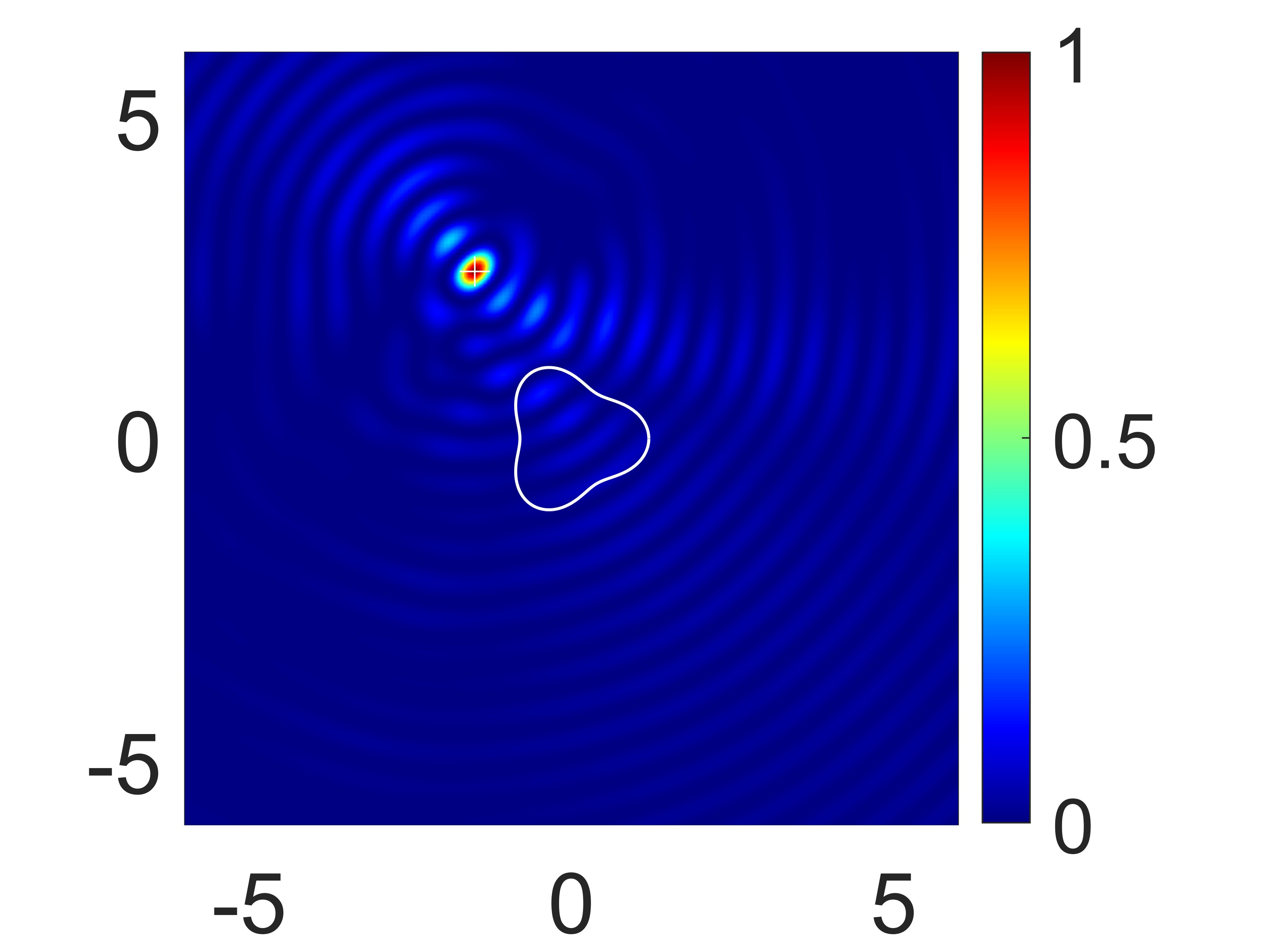}}
	\subfigure[]{\includegraphics[width=0.24\textwidth]{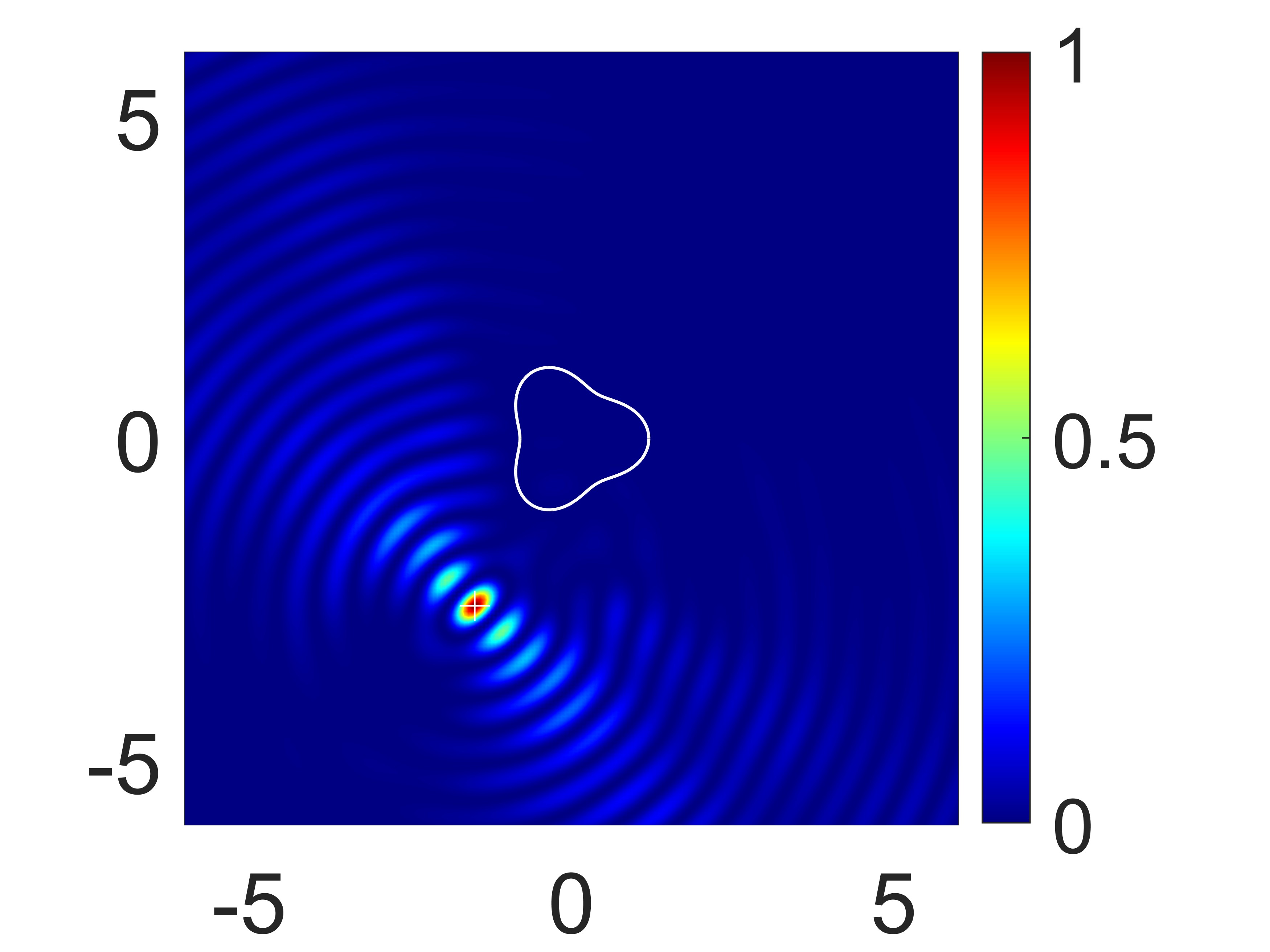}}
	\subfigure[]{\includegraphics[width=0.24\textwidth]{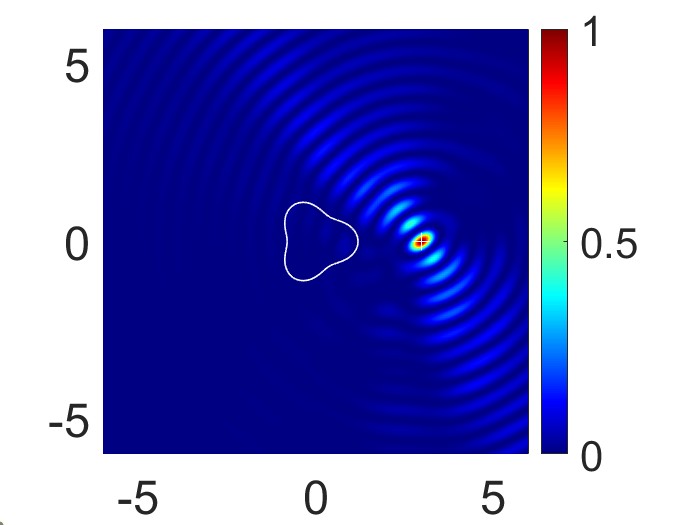}}
	\subfigure[]{\includegraphics[width=0.24\textwidth]{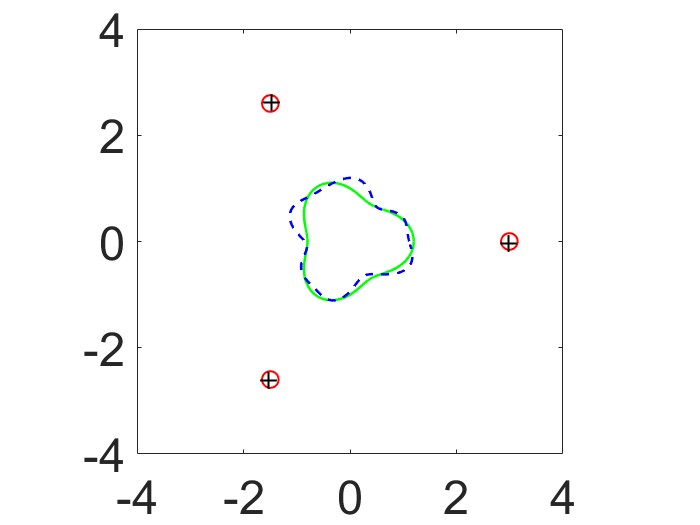}}
	\caption{Reconstruction of the obstacle and source points of different distributions.}\label{fig:pear-iosp3}
\end{figure}

\begin{figure}
	\centering
	\subfigure[]{\includegraphics[width=0.3\textwidth]{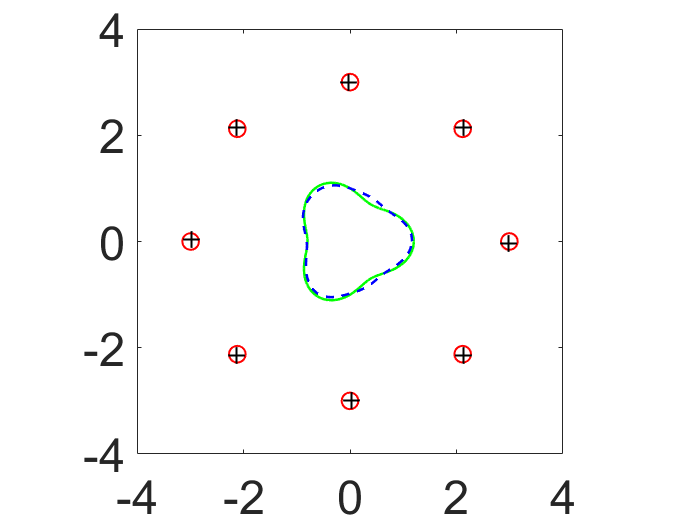}}
	\subfigure[]{\includegraphics[width=0.3\textwidth]{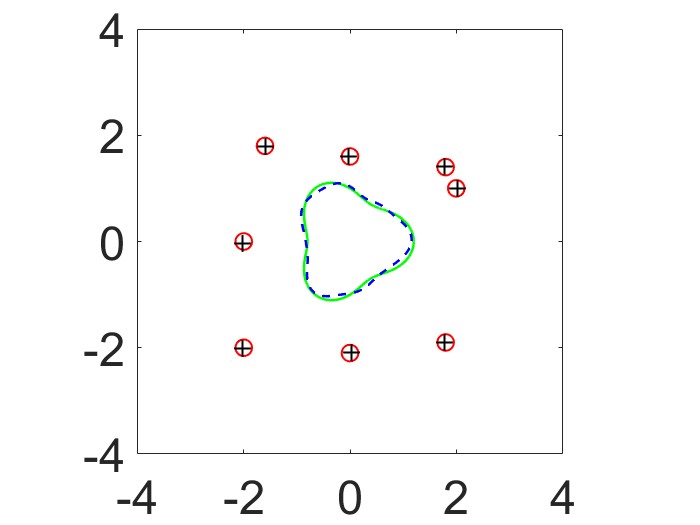}}
	\subfigure[]{\includegraphics[width=0.3\textwidth]{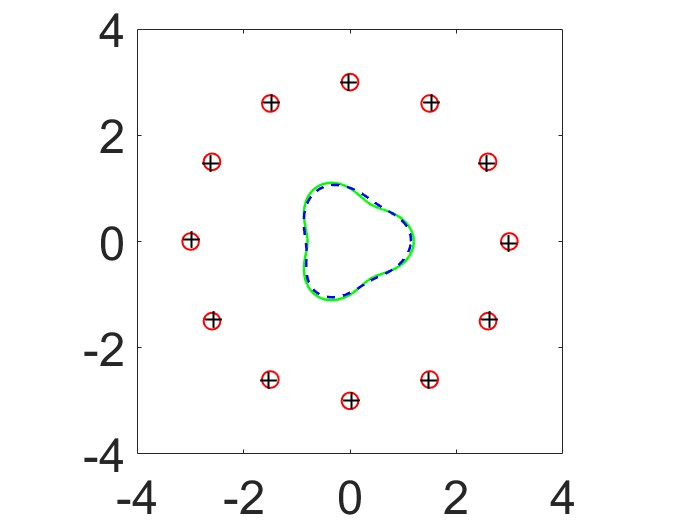}} \\
	\subfigure[]{\includegraphics[width=0.3\textwidth]{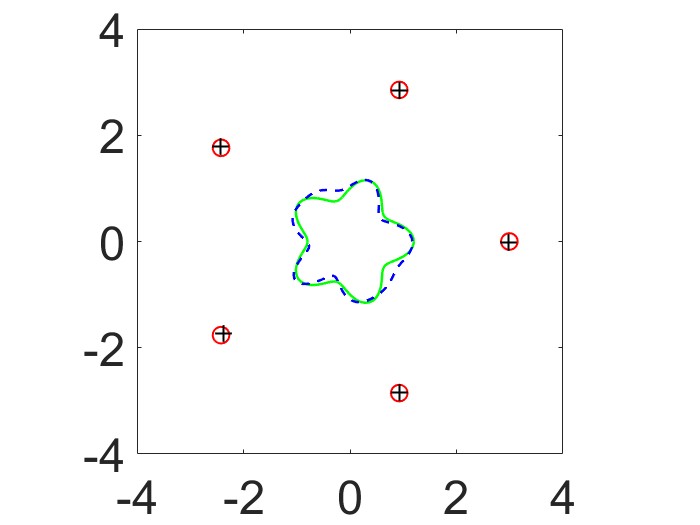}}	
	\subfigure[]{\includegraphics[width=0.3\textwidth]{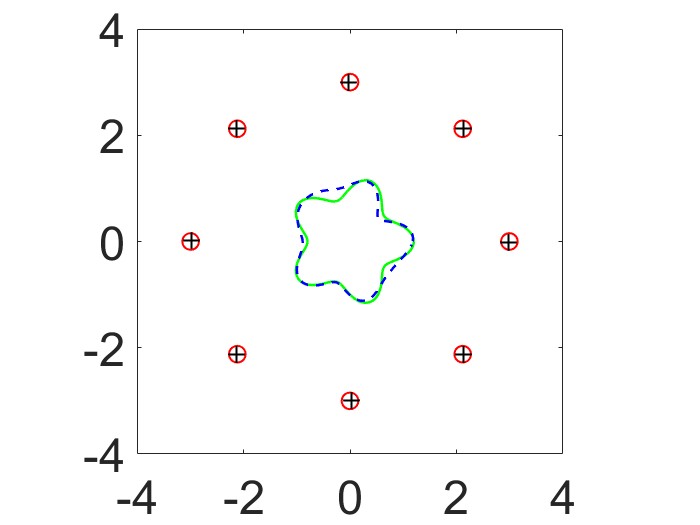}}	
	\subfigure[]{\includegraphics[width=0.3\textwidth]{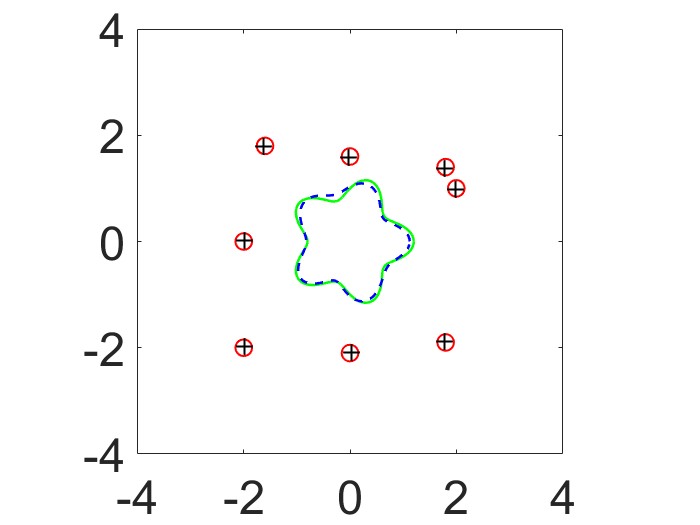}}
	\caption{Reconstruction of the obstacle and source points of different distributions. Top row: $\omega=6$; Bottom row: $\omega=9$.}\label{fig:pear-iosp}
\end{figure}

\begin{figure}
	\centering
       \subfigure[]{\includegraphics[width=0.3\textwidth]{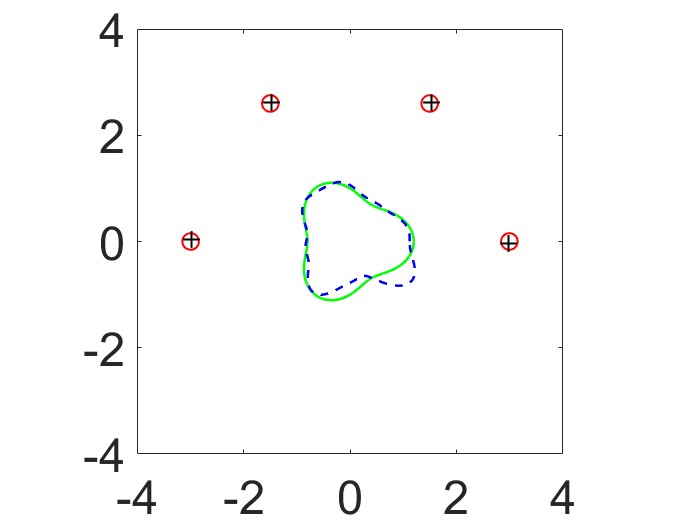}}
     \subfigure[]{\includegraphics[width=0.3\textwidth]{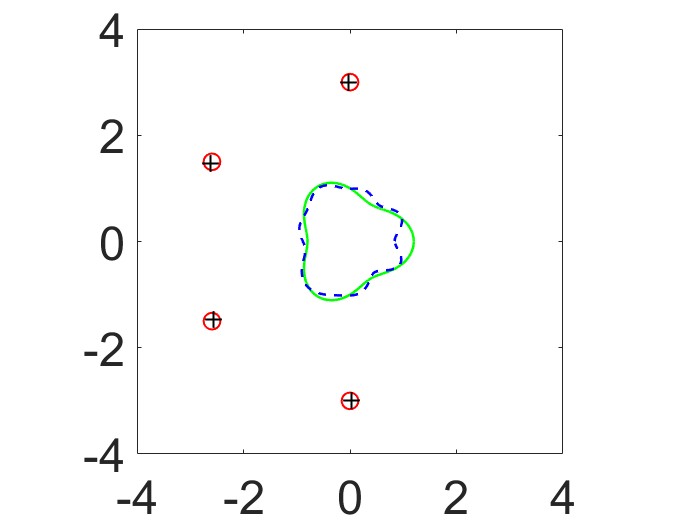}}
     \subfigure[]{\includegraphics[width=0.3\textwidth]{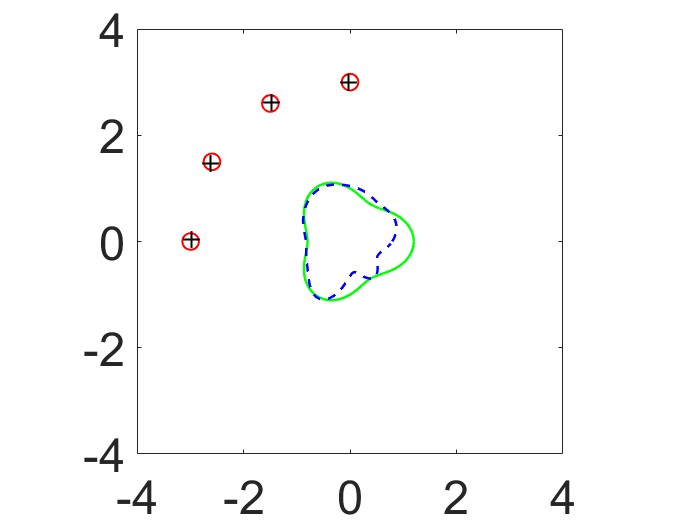}}
\caption{Reconstruction of the obstacle and source points with partial illuminations.}\label{fig:pear-iosp2}
\end{figure}

\begin{table}
	\caption{Reconstruction of the polarizations in \Cref{fig:pear-iosp}}\label{tab:polar}
	\centering
	\begin{tabular}{c|ccc}
		\toprule
		         Figure          & \ref{fig:pear-iosp}(a) & \ref{fig:pear-iosp}(b) & \ref{fig:pear-iosp}(c) \\
		$\tilde{\bm{q}}$ &   $(0.8088,0.5881)$    &   $(0.8172, 0.5763)$   &   $(0.8104, 0.5859)$    \\
		\midrule
		         Figure          & \ref{fig:pear-iosp}(d) & \ref{fig:pear-iosp}(e) & \ref{fig:pear-iosp}(f) \\
		$\tilde{\bm{q}}$ &   $(0.8083, 0.5887)$    &   $(0.8099, 0.5881)$    &   $(0.7992, 0.6060)$    \\
		\bottomrule
	\end{tabular}
\end{table}

\begin{example}
	This example concerns the reconstruction of the kite-shaped obstacle and the corresponding sources. The reconstructions are depicted in \Cref{fig:kite-iosp} to illustrate the influence of the number of source points. These results indicate that the source locations can always be favorably identified. In addition, the shape of the obstacle can be also satisfactorily recovered except for the case that only a few illuminating sources are available, for instance, see \Cref{fig:kite-iosp}(d). We would like to emphasize that, contrary to the typical inverse scattering problems where the incident waves usually can be artificially deployed to cater to reconstructions,  the source points here are in general not at our disposal in the co-inversion problem. Hence the less accurate reconstruction such as \Cref{fig:kite-iosp}(d) is due to the insufficient amount of information and the lack of data may inevitably occur. 

\end{example}

\begin{figure}
	\centering
	\subfigure[]{\includegraphics[width=0.3\textwidth]{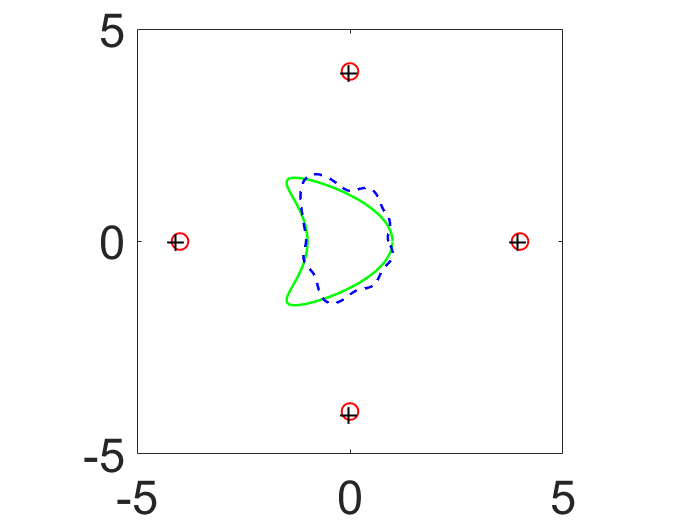}}
	\subfigure[]{\includegraphics[width=0.3\textwidth]{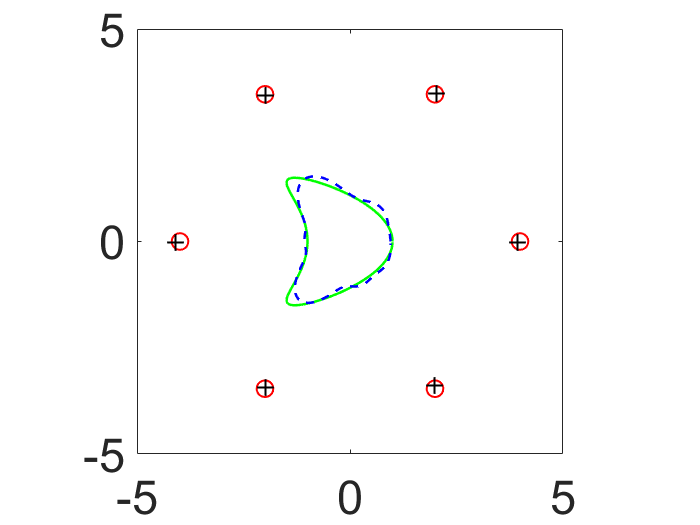}}
	\subfigure[]{\includegraphics[width=0.3\textwidth]{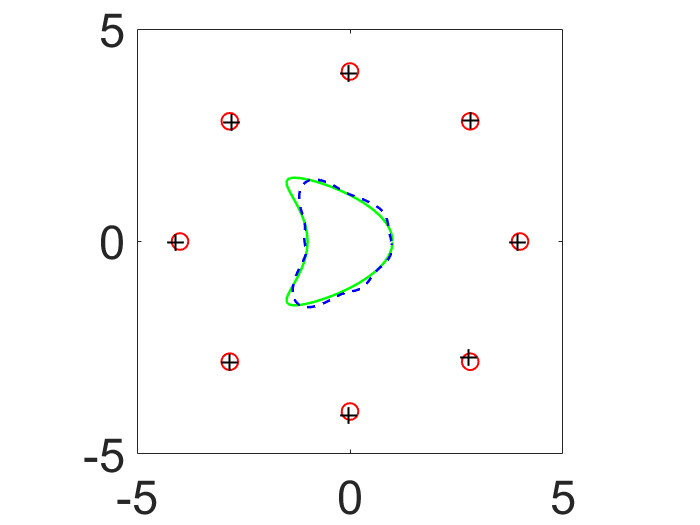}}
	\caption{Coinversion of the kite and source points with $\omega=6$. }\label{fig:kite-iosp}	
\end{figure}

We list the computational CPU time for the reconstructions in \Cref{table:CPU} to quantify the efficiency of the proposed algorithm. All of the codes in the numerical experiments are written in MATLAB and run on an Intel Core 2.6 GHz laptop. The computational cost is low because we do not need to solve any forward problem in each iteration. More importantly, this sheds light on the feasibility of a computationally affordable extension of our algorithm to 3D reconstructions.


\begin{table}
	\caption{Computing time for the reconstructions of the obstacles in \Cref{fig:pear-iosp} and  \Cref{fig:kite-iosp}}\label{table:CPU}
	\centering
	\begin{tabular}{c|cccccc}
		\toprule
		Figure & \ref{fig:pear-iosp}(a) & \ref{fig:pear-iosp}(b) & \ref{fig:pear-iosp}(c)  & \ref{fig:kite-iosp}(a) & \ref{fig:kite-iosp}(b) & \ref{fig:kite-iosp}(c) \\ \hline
		CPU(s) &          29.17          &          74.32          &          81.61         &    39.43               &      38.59            &  39.35                   \\
		iter  &           29            &           75            &           81           &     38                  &    37                  &38                       \\
		\bottomrule
	\end{tabular}
\end{table}

\begin{example}
	In the last example, we test the co-inversion of the non-symmetric obstacle and its excitation source points from the noisy total field. The non-symmetric obstacle is created as the procedure described in Example~5.1.2. The reconstructions with $\omega=8$ are displayed in \Cref{fig:suijiIOSP}.
\end{example}

\begin{figure}
	\centering
	\subfigure[]{\includegraphics[width=0.3\textwidth]{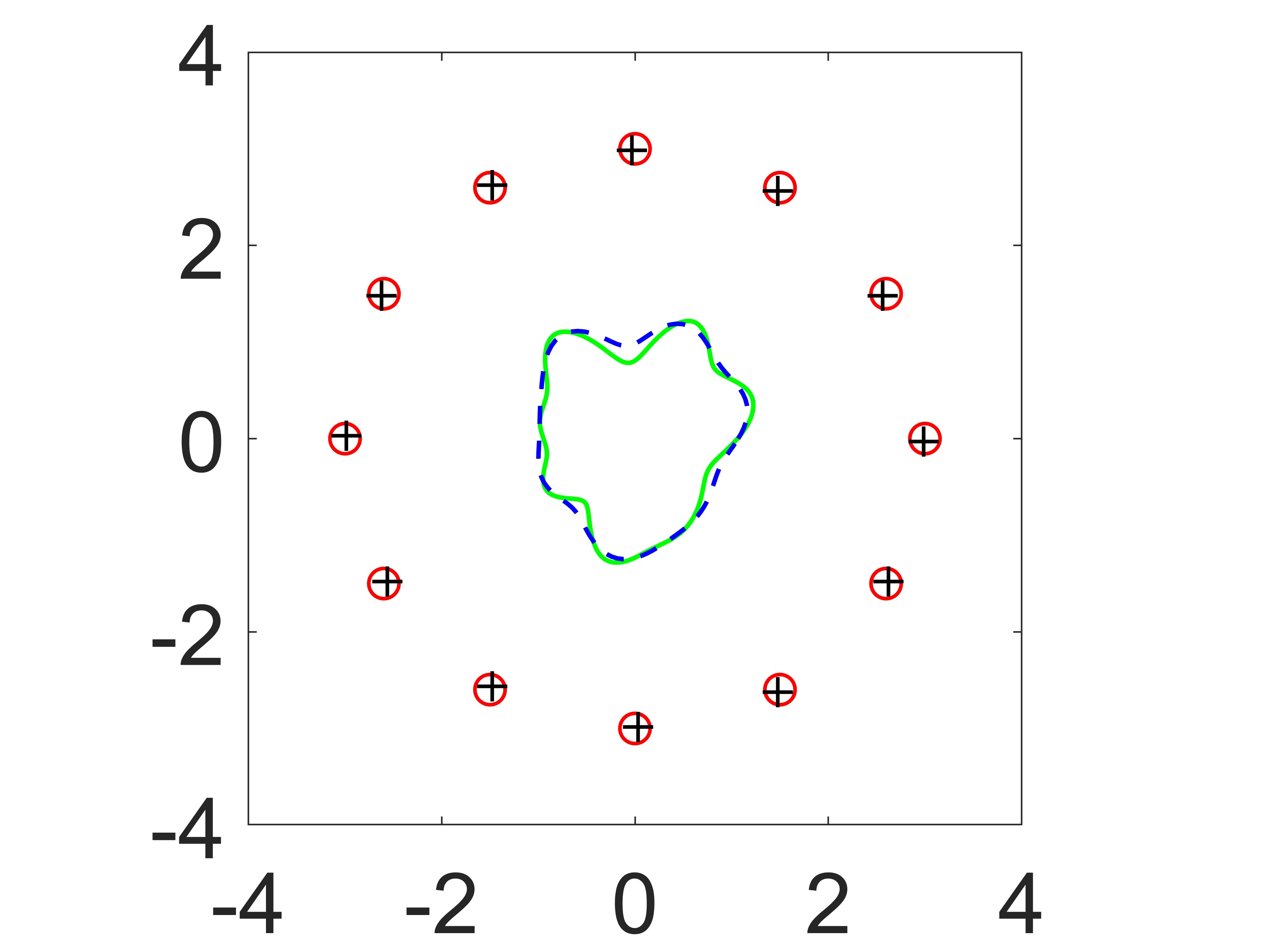}}
	\subfigure[]{\includegraphics[width=0.3\textwidth]{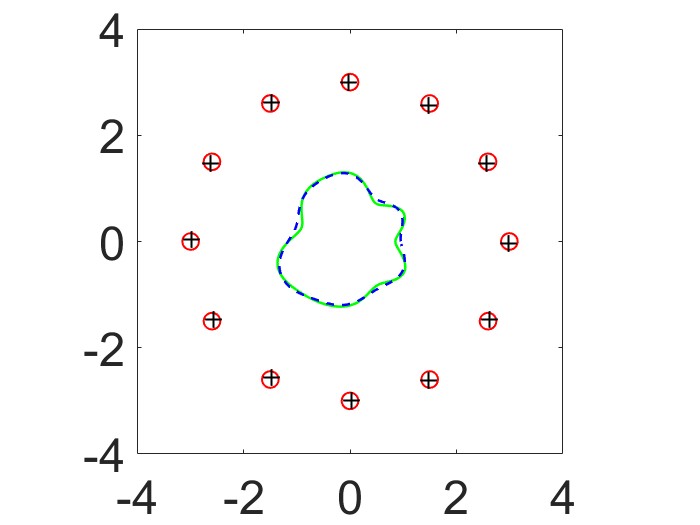}}
	\subfigure[]{\includegraphics[width=0.3\textwidth]{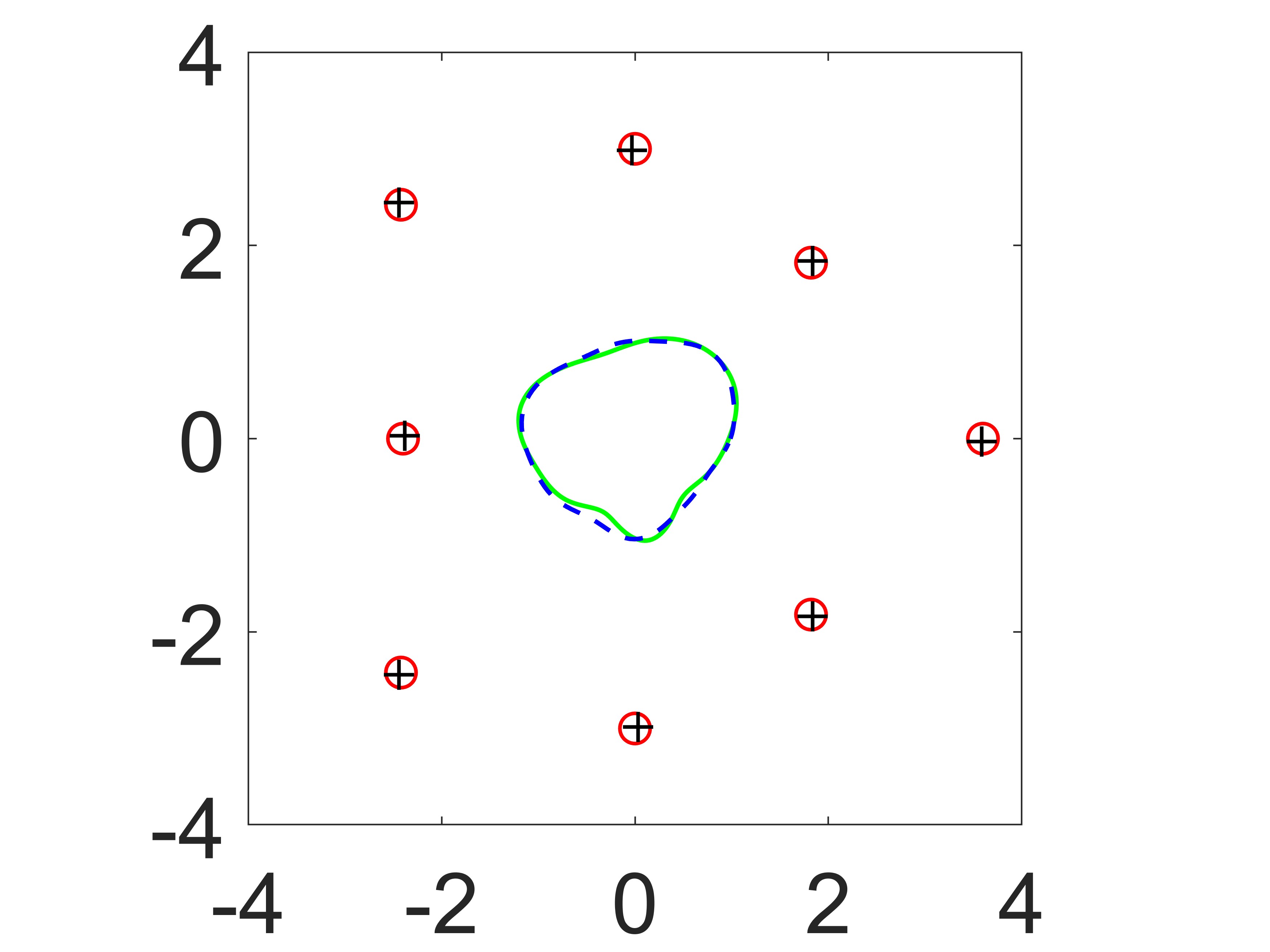}}
	\subfigure[]{\includegraphics[width=0.3\textwidth]{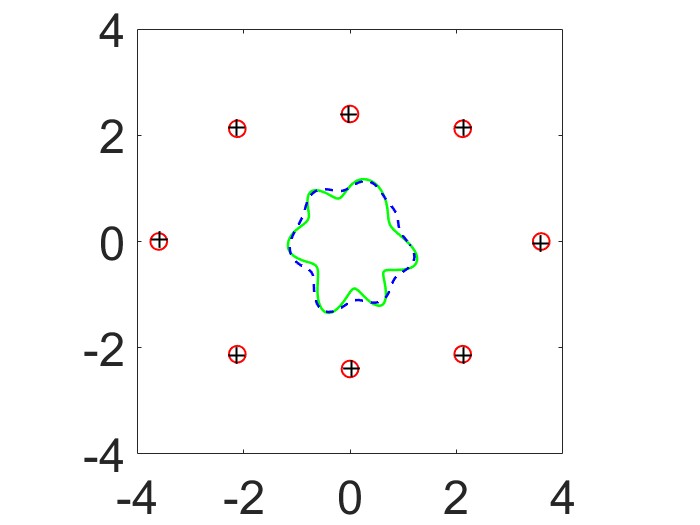}}
	\subfigure[]{\includegraphics[width=0.3\textwidth]{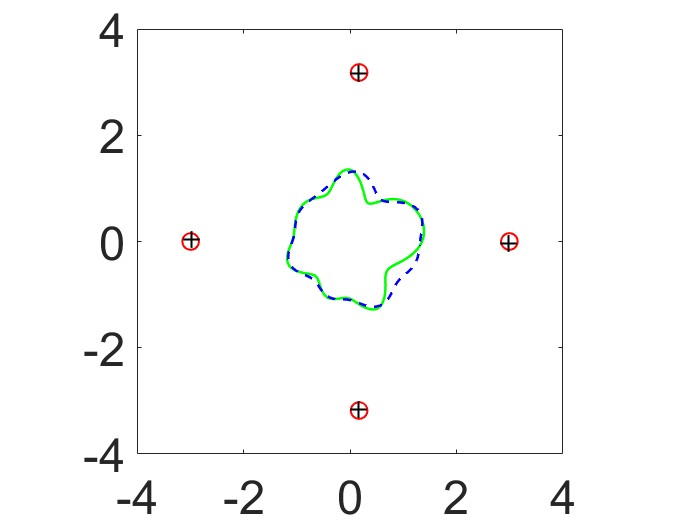}}
	\subfigure[]{\includegraphics[width=0.3\textwidth]{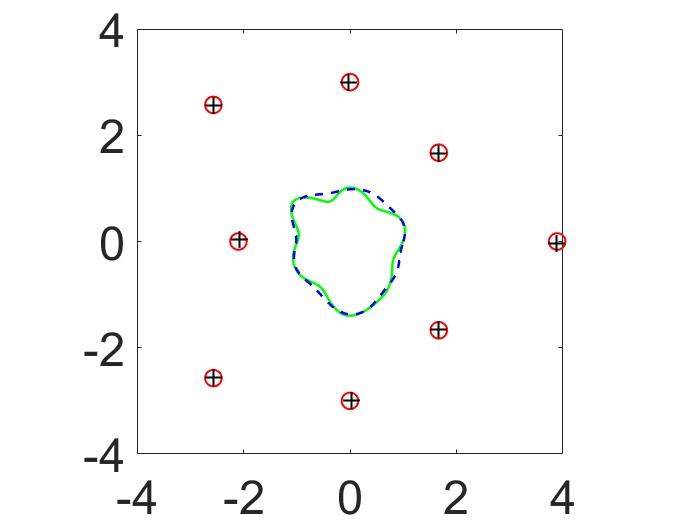}}
	\caption{Reconstruction of the non-symmetric obstacle and source points.}\label{fig:suijiIOSP}
\end{figure}

As can be seen in \Cref{fig:suijiIOSP}, all the source points are reconstructed well. For the obstacle, the convex part can be better recovered compared with the concave part because the convex part can be better illuminated.


\section{Conclusions}\label{sec:conclusion}
We propose a novel method for the elastic co-inversion problem of determining the rigid obstacle together with its excitation point sources from the measured total field. A direct sampling method is proposed to recover the source points and the associated polarization directions from the total field measurements.  After reformulating the co-inversion problem into the inverse obstacle subproblem by subtracting the incident wave due to the reconstructed source points from the total field, we propose a novel Newton-type method by approximating the scattered field as suitable layer potentials. Theoretically, we analyze the uniqueness of the co-inversion, and the indicating behaviors of the indicator functions and give an explicit formula for the shape derivative. We remark that the indicator functions and the Newton-type method are not only used in the co-inversion problem but also useful in their own right for solving inverse source and inverse scattering problems. A noteworthy advantage of our algorithm is the complete independence of any solution to the forward problem, thus our method is easy-to-implement and fast. Finally, several numerical experiments are conducted and the results show that the proposed sampling-iterative method performs well in the simultaneous reconstruction of the sources and the obstacle from the noise total field data.

We believe that the basic idea of this work can be applied to many other similar inverse scattering models, for instance, imaging in acoustics and electromagnetism. Due to the limited time and computing resources, three-dimensional numerical experiments are currently not conducted. In addition, theoretical issues such as the convergence of the iteration method have not yet been mathematically analyzed. Our ongoing and future works would consist of further attempts at these related directions. 

\section*{Acknowledgments}

Yan Chang and Yukun Guo are supported by NSFC grant 11971133. Hongyu Liu is supported by the Hong Kong RGC General Research Funds (projects 12302919, 12301420, and 11300821),  the NSFC/RGC Joint Research Fund (project N\_CityU 101/21), the France-Hong Kong ANR/RGC Joint Research Grant, A-CityU203/19. Deyue Zhang is supported by NSFC grant 12171200.

\section*{Appendix: Multiplicities of the Dirichlet eigenvalues for $-\Delta^*$ inside a ball}
\renewcommand\theequation{A.\arabic{equation}}

This appendix presents the multiplicities of the Dirichlet eigenvalues for the negative Lam\'{e} operator inside $B_R$. For $n=0,1, 2, \cdots$, denote by $J_n$ and $j_n$ the Bessel function and spherical Bessel function of order $n$, respectively. 

\begin{lemma}
	 For $r=|x|$ and $n=0,1,2,\cdots$, let
		\begin{align}
		\label{eq:Pn} &  P_n(x) =
		\begin{vmatrix}
			k_p r J'_n(k_p r) & \mathrm{i}n J_n(k_s r) \\[6pt] 
			\mathrm{i}n J_n(k_p r) & -k_s r J'_n(k_s r)
		\end{vmatrix}, \\
		\label{eq:Qn} & Q_n(x)=
		\begin{vmatrix}
			k_p r j'_n(k_p r) & n(n+1) j_n(k_s r) \\[6pt] 
			j_n(k_p r) & j_n(k_s r)+k_s r j'_n(k_s r)
		\end{vmatrix}.
	\end{align}
   Then the sum of the multiplicities of the Dirichlet eigenvalues for $-\Delta^*$ inside $B_R$ is given by
	\begin{align}\label{eq:N}
	N_0:=\left\{
	\begin{aligned}
	&\sum_{p_{nl}<R}(2n+1),\quad & d=2,\\
	&\sum_{t_{nl}<k_s R}(2n+1)+\sum_{q_{nl}<R}(2n+1),\quad & d=3,
	\end{aligned}
	\right.
	\end{align} 
    where $p_{nl}$, $q_{nl}$ and $t_{nl}(l=0,1,\cdots; n=0,1,\cdots)$ are respectively the $l$-th positive zero of $P_n, Q_n$, and $j_n$, namely, $P_n(p_{nl})=Q_n(q_{nl})=j_n(t_{nl})=0$.
\end{lemma}

\begin{proof}
	We consider the Dirichlet eigenvalue problem inside $B_R\subset \mathbb{R}^d$:
	\begin{equation}\label{eq:interior}
     \begin{cases}
	\Delta^*\bm{v}+\omega^2\bm{v}=\bm{0}, & \text{in}\ B_R, \\
	\hspace{1.5cm}\bm{v}=\bm{0},  & \text{on}\ \Gamma_R.
	\end{cases}
	\end{equation}

	(i) $d=2$. In the 2D case, the eigenfunction $\bm{v}$ of \eqref{eq:interior} can be split by \eqref{eq:split} into $\bm{v}=\bm{v}_p+\bm{v}_s=\nabla \phi+\mathbf{curl}\psi$ where the scalar functions $\phi$ and $\psi$ can be given by 
	\begin{equation}
	\phi(x)=\sum_{n=-\infty}^\infty \phi_n J_n(k_p|x|)\mathrm{e}^{\mathrm{i}n\theta},\quad
	\psi(x)=\sum_{n=-\infty}^\infty \psi_n J_n(k_s|x|)\mathrm{e}^{\mathrm{i}n\theta}.
	\end{equation}
   with the polar coordinate $x=|x|(\cos\theta,\sin\theta)^\top$ and the coefficients $\phi_n$ and $\psi_n$, respectively. 
   
	Define $\bm{e}_\rho=(\cos\theta,\sin\theta)^\top$ and $\bm{e}_\theta=(-\sin\theta,\cos\theta)^\top$, then from
	$$
	\nabla w=\frac{\partial w}{\partial \rho}\bm{e}_\rho+\frac{1}{|x|}\frac{\partial w}{\partial\theta}\bm{e}_\theta,\quad 
	\mathbf{curl} w=\frac{1}{|x|}\frac{\partial w}{\partial\theta}\bm{e}_\rho-\frac{\partial w}{\partial \rho}\bm{e}_\theta,
	$$
	we have 
	\begin{align*}
	& \bm{v}_p(x)=\sum_{n=-\infty}^\infty \phi_n\nabla\left(J_n(k_p|x|)\mathrm{e}^{\mathrm{i}n\theta}\right)=\sum_{n=-\infty}^\infty\phi_n\left(k_pJ'_n(k_p|x|)\mathrm{e}^{\mathrm{i}n\theta}\bm{e}_\rho+\frac{\mathrm{i}n}{|x|}J_n(k_p|x|)\mathrm{e}^{\mathrm{i}n\theta}\bm{e}_\theta\right),\\
	& \bm{v}_s(x)=\sum_{n=-\infty}^\infty \psi_n\mathbf{curl}\left(J_n(k_s|x|)\mathrm{e}^{\mathrm{i}n\theta}\right)=\sum_{n=-\infty}^\infty\psi_n\left(\frac{\mathrm{i}n}{|x|}J_n(k_s|x|)\mathrm{e}^{\mathrm{i}n\theta}\bm{e}_\rho-k_s J'_n(k_s|x|)\mathrm{e}^{\mathrm{i}n\theta}\bm{e}_\theta\right).
	\end{align*}
By the Dirichlet boundary condition $(\bm{v}_p+\bm{v}_s)|_{\Gamma_R}=\bm{0}$ and the orthogonality $\bm{e}_\rho\cdot\bm{e}_\theta=0$, we get	  
$$
\begin{cases}
	k_p R J'_n(k_pR)\phi_n+\mathrm{i}n J_n(k_sR)\psi_n=0, \\
	\mathrm{i}n J_n(k_pR)\phi_n-k_s R J'_n(k_sR)\psi_n=0,
	\end{cases}
	\quad \forall n=1, 2, \cdots.
$$
	Further, by \cite[Theorem 2.7]{CK19}, the multiplicity of the Dirichlet eigenvalues is given by
	\[
	\sum_{q_{nl}<R}(2n+1),
	\]
	with $q_{nl}$ the $l$-th zero of $P_n$ as defined in \eqref{eq:Pn}.
		
	(ii) $d=3$. Let $\hat{x}=x/|x|\in\mathbb{S}^2$ and $Y_n^m(\hat{x})(m=-n,\cdots, n; \,n=0,1,2,\cdots)$ be the spherical harmonics \cite{CK19}. Similar to the 2D case, we introduce
	\begin{align*}
	& \bm{v}_p(x):=\sum_{n=1}^\infty\sum_{m=-n}^na_n^m\nabla(j_n(k_p|x|)Y_n^m(\hat{x})),\\
	& \bm{v}_s(x):=\sum_{n=1}^\infty\sum_{m=-n}^n\left(b_n^mM_n^m(|x|,\hat{x})+c_n^m\nabla\times M_n^m(|x|,\hat{x})\right),
	\end{align*}
	where $a_n^m,\,b_n^m, c_n^m$ are the coefficients and $M_n^m(|x|,\hat{x})=\nabla\times\left(x j_n(k_s |x|)Y_n^m(\hat{x})\right).$ For later analysis, we explicitly rewrite $\bm{v}$ as follows:
	\begin{align*}
	\bm{v}(x)=&\sum_{n=1}^\infty\sum_{m=-n}^na_n^m\left(k_p j'_n(k_p|x|)Y_n^m(\hat{x})\hat{x}+\frac{j_n(k_p|x|)}{|x|}\text{Grad}Y_n^m(\hat{x})\right)\\
	&+b_n^m j_n(k_s|x|)\text{Grad}Y_n^m(\hat{x})\times\hat{x}\\
	&+c_n^m\left(n(n+1)\frac{j_n(k_s|x|)}{|x|}Y_n^m(\hat{x})\hat{x}+\left(\frac{j_n(k_s|x|)}{|x|}+k_s j'_n(k_s|x|)\right)\text{Grad}Y_n^m(\hat{x})\right),
	\end{align*}	
	where $\mathrm{Grad}$ is the surface gradient.	
	
	From the Dirichlet condition $(\bm{v}_p+\bm{v}_s)|_{\Gamma_R}=\bm{0}$, one can deduce that 
	\begin{equation}\label{eq:abc}
	\left\{
	\begin{aligned}
	& a_n^m k_p j'_n(k_p R)+n(n+1)c_n^m\frac{j_n(k_sR)}{R}=0,\\
	& a_n^m\frac{j_n(k_p R)}{R}+c_n^m\left(\frac{j_n(k_sR)}{R}+k_s j'_n(k_sR)\right)=0,\\
	& b_n^m j_n(k_sR)=0,
	\end{aligned}
	\right.\quad\forall m=-n,\cdots, n; \,n=0,1,2,\cdots
	\end{equation}
	From \eqref{eq:abc}, we see that the multiplicity of the Dirichlet eigenvalues is given by
	$$
	\sum_{t_{nl}<k_s R}(2n+1)+\sum_{q_{nl}<R}(2n+1),
	$$
	such that $j_n(t_{nl})=0$ and $Q_n(q_{nl})=0$ with $Q_n$ defined by \eqref{eq:Qn}.
\end{proof}


\end{document}